\documentclass[leqno]{article}

\usepackage{ifluatex}
\ifluatex%
\usepackage{luatex85}
\else%
\usepackage[utf8]{inputenc}
\fi

\usepackage[english]{babel}
\usepackage[pass]{geometry}
\usepackage{microtype}
\usepackage{amsfonts}
\usepackage{amssymb}
\usepackage{empheq}
\usepackage{latexsym}
\usepackage{amsmath}
\usepackage{amsthm}
\usepackage{array}
\usepackage{bm}
\usepackage{dsfont}
\usepackage{pifont}
\usepackage{esint}
\usepackage{stmaryrd}
\usepackage{hhline}
\usepackage{accents}
\usepackage{enumitem}
\usepackage{tabularx}
\usepackage{booktabs}
\usepackage{csquotes}
\usepackage{xcolor}
\usepackage{graphicx}
\usepackage{tcolorbox}
\usepackage{tabularx}
\usepackage{colortbl}
\usepackage{titlesec}
\usepackage{url}
\usepackage[hidelinks,breaklinks]{hyperref}
\usepackage{subfig}
\usepackage{wrapfig}
\usepackage{float}
\usepackage{caption}
\usepackage{tikz}
\usepackage{authblk}
\usepackage[hang,flushmargin]{footmisc}
\usepackage[backend=biber]{biblatex}
\newtheorem{theorem}{Theorem}[section]
\newtheorem{proposition}[theorem]{Proposition}
\newtheorem{assumption}[theorem]{Assumption}

\newtheorem{lemma}[theorem]{Lemma}
\newtheorem{remark}[theorem]{Remark}
\newtheorem{example}[theorem]{Example}

\allowdisplaybreaks

\newcommand\blfootnote[1]{%
  \begingroup
  \renewcommand\thefootnote{}\footnote{#1}%
  \addtocounter{footnote}{-1}%
  \endgroup
}

\SetSymbolFont{stmry}{bold}{U}{stmry}{m}{n}

\author[1]{David Lannes}
\author[1]{Mathieu Rigal}
\affil[1]{Institut de Mathématiques de Bordeaux, Université de Bordeaux et CNRS UMR 5251, 351 Cours de la libération 33405 Talence Cedex, France}

\date{February 6, 2024}

\begin{document}

\title{General boundary conditions for a Boussinesq model with varying bathymetry}

\maketitle\thispagestyle{empty}

\begin{abstract}
  This paper is devoted to the theoretical and numerical investigation of the initial boundary value problem for a system of equations used for the description of waves in coastal areas, namely, the Boussinesq-Abbott system in the presence of topography. We propose a procedure that allows one to handle very general linear or nonlinear boundary conditions. It consists in reducing the problem to a system of conservation laws with nonlocal fluxes and coupled to an ODE. This reformulation is used to propose two hybrid finite volumes/finite differences schemes of first and second order respectively. The possibility to use many kinds of boundary conditions is used to investigate numerically the asymptotic stability of the boundary conditions, which is an issue of practical relevance in coastal oceanography since asymptotically stable boundary conditions would allow one to reconstruct a wave field from the knowledge of the boundary data only, even if the initial data is not known.
\end{abstract}

\vspace{3em}
{\centering\itshape\noindent
in memory of Vassilios A. Dougalis\par
}

\vspace{3em}\noindent
{\small This work was partially supported by the Institut des Mathématiques pour la Planète Terre and by the BOURGEONS project, grant ANR-23-CE40-0014-01 of the French National Research Agency (ANR).\par}


\section{Introduction}

\subsection{General motivation and objectives}

The dynamics of water waves over a non flat bottom in the littoral area  is a subject of great interest with far-reaching societal and economic stakes. Among them, one can mention safety aspects linked to the risk of flooding and sailing accidents, the impact of waves on the morphodynamics of the beach through sedimentation and erosion, or the issue of marine renewable energies, for which a good knowledge of the waves dynamics would be beneficial to increase the yield of offshore mooring systems or to choose pertinent wave farm locations. When considering coastal flows, it is not always possible to neglect dispersive effects which play an important role in the shoaling zone and in the generation of extreme waves. In order to accurately describe these phenomena, one therefore needs a model of shallow free surface flows that takes into account both nonlinear and dispersive effects.   In this regard, the family of weakly nonlinear vertically-averaged Boussinesq-type systems is pertinent as it presents a good compromise between accuracy and reduced complexity compared for instance to the shallow water system, not accurate enough because it neglects dispersive effects, or to the fully nonlinear Serre-Green-Naghdi equations, more accurate but also more complex \cite{LannesNL2021}.

For practical applications, one usually considers
a bounded spatial domain, and one has to solve an initial boundary value problem: the solution is computed at all times in terms of the initial data in the interior of the domain, as well as of the boundary data imposed at all times on its boundary. In the presence of dispersive terms, the question of boundary conditions is difficult and widely open, yet it is crucial for applications. The purpose of this work is to investigate this issue both theoretically and numerically in the case of the one dimensional Boussinesq-Abbott model with a non flat bottom, namely,
$$
  \left\{
  \begin{array}{l}
    \partial_t \zeta + \partial_x q = 0 \\
    (1  + h_b{\mathcal T}_b)\partial_t q + \partial_x (\frac{1}{h}q^2)+{\mathtt g} h \partial_x \zeta=0
  \end{array}\right.\quad t>0,\quad x\in (0,\ell),
$$
where $\zeta(t,x)$ denotes the elevation of the fluid at time $t$ and horizontal coordinate $x$ with respect to the rest state $z=0$; $q(t,x)$ is the horizontal discharge of the fluid; $h(t,x)=h_0 + \zeta(t,x) - b(x)$ is the water height, with $h_0$ the characteristic depth and $z=-h_0+b(x)$ a parametrization of the bottom topography; ${\mathtt g}$ is the acceleration of gravity; and finally ${\mathcal T}_b$ is a positive second order differential operator, whose exact expression is not important at this point. This model was proposed by Abbott~\cite{Abbott_Petersen_Skovgaard_1978} as a variant of the original Boussinesq-Peregrine model, which is written in terms of $\zeta$ and the average velocity $\overline{v}=q/h$; see~\cite{FILIPPINI2015109} for a comparison of these two approaches.

More specifically, we propose a new method that enables to prescribe a nonlinear function of the unknowns $(\zeta,q)$ at the boundaries of the domain by a given function of time. The possibility to enforce such general non homogeneous boundary conditions is a prerequisite to be able to go beyond simple academic test-cases such as solitary waves, and instead force realistic and complex wave fields in the domain, which typically feature swell and infra-gravity waves characterized by very disparate spatial scales and multi-directional propagation. This question is also related to that of asymptotic stability in the following sense: if we know the boundary data for all times (via buoys measurements for instance) but not the initial condition, can we expect the solution computed by solving the initial boundary value problem with an arbitrary initial data to converge in time to the real solution, and how does the choice of boundary conditions impact this convergence? This is a completely open theoretical problem, for which we propose here some numerical insight based on a new second order numerical scheme.

\subsection{State of the art and contribution}

Before describing the state of the art for the mathematical and numerical analysis of the initial boundary value problem associated with nonlinear dispersive models such as the Boussinesq-Abbott equations, let us recall some elements of the hyperbolic theory. Indeed, the Boussinesq-Abbott equations are a dispersive perturbation of the nonlinear shallow water equations that can be deduced from the Boussinesq-Abbott equations by neglecting the operator ${\mathcal T}_b$, namely,
$$
\begin{cases}
\partial_t \zeta + \partial_x q=0,\\
\partial_t q + \partial_x (\frac{1}{h}q^2)+{\mathtt g} h \partial_x \zeta=0
\end{cases}
\quad t>0,\quad x\in (0,\ell).
$$
This model belongs to the class of hyperbolic systems for which initial boundary value problems are well understood (see for instance~\cite{BenzoniSerre} as well as~\cite{IguchiLannes} for a detailed analysis of the one dimensional case). As we shall see, despite the fact that the Boussinesq-Abbott system is a perturbation of the nonlinear shallow water equations, the behavior of the associated initial boundary value problems differ drastically.
When applied to the nonlinear shallow water equations, the general theory of hyperbolic systems implies that under certain well identified conditions, it is possible to solve the equations on a finite domain if we know $\zeta$ and $q$ at $t=0$, and if we impose at all times one scalar boundary data at $x=0$ and $x=\ell$: typically, one can impose $\zeta$, $q$ or the incoming Riemann invariant (see in particular V. Dougalis' contribution for the numerical analysis of characteristic boundary conditions~\cite{antonopoulos2017}). The missing boundary data is then recovered locally through the method of characteristics applied to the outgoing Riemann invariants. This is a fundamental difference compared to dispersive models, which do not admit Riemann invariants due to their non hyperbolic nature, and a consequence is that it is no longer possible to recover the missing boundary data using local arguments as pointed in~\cite{Lannes_Weynans_2020}.

V. Dougalis was a pioneer in the analysis of initial boundary value problems for Boussinesq-type equations with non trivial (i.e. non periodic) boundary conditions. In \cite{antonopoulos2009} he considered systems of the Bona-Smith family (for which a dispersive term is also present in the equation for the surface elevation) with non homogeneous boundary conditions on the surface elevation and on the velocity. In the case of homogeneous boundary conditions, he then proposed a numerical Galerkin-finite element scheme \cite{antonopoulos2010}; he even addressed  the two-dimensional case in \cite{dougalis2009,dougalis2010} where he proposed a mathematical and numerical analysis for various types of homogeneous boundary conditions, still for systems of the Bona-Smith family. In \cite{antonopoulos2012} he considered the "classical" Boussinesq system (corresponding to the Boussinesq-Peregrine system with a flat topography), with a homogeneous boundary condition on the velocity. He also considered the case of a variable topography \cite{kounadis2021} with homogeneous condition on the velocity or approximate transparent boundary conditions. The contribution of V. Dougalis and his coauthors is without doubt the most important to the theoretical and numerical analysis of the initial boundary value problem of Boussinesq-type equations. In the wake of these works, we propose here to investigate the case of non homogeneous boundary conditions.

Several authors also considered the related problem of transparent boundary conditions: for the linear KdV, BBM or Boussinesq equations, this was addressed in \cite{besse2017,besse2018,steinstraesser2019domain,kazakova2020} while for the nonlinear case approaches based on  perfectly matched layer (PML) methods have been proposed  for the KdV equation \cite{besse2022} and for a hyperbolic relaxation of the Serre-Green–Naghdi equations \cite{kazakova2018}. This approach can be adapted to handle the initial boundary value problem; instead of imposing a boundary condition, one imposes an initial data in the PML. This is also the idea behind the source function method \cite{Wei}. These methods are robust and popular but have the drawback of being only approximations of the initial boundary value problems and to require to work on a considerably larger computational domain, at the cost of computational time. Besides let us mention~\cite{parisot} where a class of boundary conditions for  time-discrete Serre-Green-Naghdi equations are considered.

It was shown recently that wave-structure interaction problems for the Boussinesq equations could be reduced to  initial boundary value problems with non homogeneous boundary conditions. A new method was proposed to handle this problem theoretically~\cite{BreschLannesMetivier,BeckLannes} and numerically~\cite{beck2023numerical}; it was also generalized in~\cite{Lannes_Weynans_2020} to handle generating boundary conditions (the boundary condition is on the surface elevation $\zeta$). All these references considered the case of the Boussinesq-Abbott model with a flat topography, and with boundary conditions imposed on $\zeta$ or $q$. For applications to coastal oceanography it is however important to consider a non-flat topography, as well as other kinds of boundary conditions. The goal of this article is to generalize the above method to propose a solution to these two problems, and to design a general second order well-balanced scheme to approximate the solutions.

\subsection{Organization of the paper}

We first propose in Section~\ref{sec:models-BC} a theoretical analysis of the initial boundary value problem associated with the Boussinesq-Abbott when the boundary conditions are imposed on the surface elevation at $x=0$ and $x=\ell$. In the case of a flat bottom, considered in~\S \ref{sect:flat} this is a generalization of \cite{Lannes_Weynans_2020} to the case of a finite interval; in the case of a non-flat topography studied in~\S \ref{subsec:boussinesq-peregrine}, the analysis is significantly more involved because some crucial commutation properties used in~\cite{Lannes_Weynans_2020} are no longer true. Well-posedness is proved using a reformulation of the original initial boundary value problem as a simple initial value problem (no boundary condition); a well-balanced version of this formulation is proposed in~\S \ref{subsec:unify-models-wb}, on which our numerical scheme is based on.

We then consider the case of general boundary conditions in Section~\ref{subsec:generic-bnd}. We work in a general framework where a (possibly nonlinear) function of $\zeta$ and $q$ is imposed at each endpoint of the domain; these functions are called \emph{input} functions. There exist also in this framework so called \emph{output} functions and a reconstruction mapping that allows to recover $\zeta$ and $q$ from the knowledge of the input and output functions, see~\S \ref{subsectoutput}. We show in~\S \ref{secteqout} that the output functions can be recovered through the resolution of a first or second order ODE; the number of initial conditions that must be provided to solve this ODE depends of course on its order; in addition compatibility conditions must also be imposed on the data. The analysis of these two types of conditions is performed in~\S \ref{sectcompat}. The well-posedness of the original initial boundary value problem with general boundary conditions can then be established in  \S \ref{sectWP}. We end this section by commenting in~\S \ref{sectAS} on an open theoretical problem of high practical interest: the asymptotical stability of the boundary conditions. In other terms, the question is to know whether the solution of the initial boundary value problem can be asymptotically recovered after a transitory regime if we start from a different initial data but impose the same boundary data. 

We then present in Section~\ref{sectschemes} the numerical schemes we use in our numerical simulations. We use an hybrid finite volumes/finite differences approach. We propose both a first order Lax-Friedrichs scheme in~\S \ref{subsec:1st-order-discr} and a second order MacCormack scheme in~\S \ref{subsec:2nd-order-discr}.

Numerical simulations are finally presented in Section \ref{sectnum}. We first show in~\S \ref{sectnumres} the ability of our numerical schemes to solve non trivial initial boundary value problems in the presence of topography. Several kinds of boundary conditions are investigated. We finally investigate in \S \ref{sectnumAS} the issue of asymptotic stability. It is in particular shown that numerical asymptotic stability is achieved by imposing nonlinear boundary conditions (the Riemann invariants associated with the nonlinear shallow water equations), but not if we impose the surface elevation or the discharge at the boundaries.  This shows the relevance of being able to enforce generic nonlinear boundary conditions when dealing with complex applications of initial boundary value problems.

\medbreak
\noindent
{\bf Acknowledgement.} The authors warmly thank Philippe Bonneton for many fruitful discussions and for his insights on the physical motivations of this work.

\section{The Boussinesq-Abbott system with boundary conditions on the surface elevation}
\label{sec:models-BC}

We propose here a theoretical analysis of the initial boundary value problem for the Boussinesq-Abbott system with boundary conditions on the surface elevation. We first treat in \S \ref{sect:flat} the case of flat bathymetries; this problem was considered in \cite{Lannes_Weynans_2020} on the half-line; we extend it here to treat the case of a finite interval, with boundary conditions on the surface elevation at both ends. We then consider in~\S \ref{subsec:boussinesq-peregrine} the Boussinesq-Abbott equations with topography. As in the case of flat topography, it is possible to reduce the problem to an initial boundary value problem, which we use to prove well-posedness, but the analysis is considerably more difficult because the operator ${\mathcal T}_b$ does not commute with space derivatives. In order to prepare the ground for our numerical scheme, a well balanced version of this reformulation is proposed in \S \ref{subsec:unify-models-wb}.

\subsection{The case of flat bathymetries}\label{sect:flat}

We consider here the case of a flat bathymetry; the Boussinesq-Abbott model  then reduces to 
\begin{align}\label{eq:BA}
  \left\{
  \begin{array}{l}
    \partial_t \zeta + \partial_x q = 0 \\
    (1  - \frac{h_0^2}{3}\partial_x^2)\partial_t q + \partial_x f_{\text{NSW}} = 0
  \end{array}\right.\quad \text{in}\ (0,\ell)\ ,
\end{align}
where we recall that 
$$ 
f_{\text{NSW}}=f_{\rm NSW}(\zeta,q)=\frac{1}{2} {\mathtt g} (h^2-h_0^2) +\frac{1}{h}q^2,
$$
where $h=h_0+\zeta$. The initial boundary value problem for this system was studied in \cite{Lannes_Weynans_2020} on the half-line $(0,\infty)$, with a prescribed boundary value on the surface elevation $\zeta$ at $x=0$ (see also \cite{beck2023numerical} for the case of a prescribed boundary value on $q$). We extend this result here to the case of a finite interval $(0,\ell)$ with prescribed boundary values on $\zeta$ at $x=0$ and $x=\ell$.

\subsubsection{Notations}\label{subsec:reformulation-BL}
We first need to introduce $R_0$,  the inverse of $   (1  - \frac{h_0^2}{3}\partial_x^2)$ with homogeneous Dirichlet boundary data at $x=0$ and $x=\ell$; more precisely, we set
\begin{align}
\label{defR0}
  R^0 :f\in L^2(0,\ell)\longmapsto u\in H^2(0,\ell)\quad \text{solving}\quad
  \left\{
  \begin{array}{l}
    \big[1 - \frac{h_0^2}{3} \partial_x^2 \big] u = f \\ u(0) = u(\ell) = 0
  \end{array}\right.
\end{align}
(note that $R^0$ is also well defined on $H^{-1}(0,\ell)$ with values in $H^1(0,\ell)$).
Since, contrary to \cite{Lannes_Weynans_2020}, we work on a finite interval, we also need to introduce ${\mathfrak s}^{(0)}$ and ${\mathfrak s}^{(\ell)}$ the solutions of the equations of the homogeneous equation but with non homogeneous boundary conditions at the left and right boundaries respectively,
\begin{equation}\label{defs0l}
  \left\{
  \begin{array}{l}
    \big[1 - \frac{h_0^2}{3} \partial_x^2 \big] {\mathfrak s}^{(0)} = 0 \\ {\mathfrak s}^{(0)}(0) = 1, \qquad {\mathfrak s}^{(0)}(\ell) = 0
  \end{array}\right.
\quad\mbox{ and }\quad
 \left\{
  \begin{array}{l}
    \big[1 - \frac{h_0^2}{3} \partial_x^2 \big] {\mathfrak s}^{(\ell)} = 0 \\ {\mathfrak s}^{(\ell)}(0) = 0, \qquad {\mathfrak s}^{(\ell)}(\ell) = 1
  \end{array}\right. .
\end{equation}
We also denote by ${\mathfrak S}'$ the matrix
\begin{equation}\label{defSprime}
{\mathfrak S}'=\begin{pmatrix}
({\mathfrak s}^{(0)})'(0)  & ({\mathfrak s}^{(\ell)})'(0) \\
({\mathfrak s}^{(0)})'(\ell)  & ({\mathfrak s}^{(\ell)})'(\ell) 
\end{pmatrix}.
\end{equation}
\begin{remark}
One can compute ${\mathfrak s}^{(0)} $ and ${\mathfrak s}^{(\ell)}$ explicitly, namely,
\begin{equation}\label{eq:s-sinh}
{\mathfrak s}^{(0)} (x) =\frac{\sinh(\sqrt{3}\frac{(\ell-x)}{h_0})}{\sinh(\sqrt{3}\frac{\ell}{h_0})},
\qquad
{\mathfrak s}^{(\ell)}(x)=  \frac{\sinh(\sqrt{3}\frac{x}{h_0})}{\sinh(\sqrt{3}\frac{\ell}{h_0})},
\end{equation}
from which an explicit expression for ${\mathfrak S}'$ easily follows.
\end{remark}
We finally introduce the inverse of $   (1  - \frac{h_0^2}{3}\partial_x^2)$ with homogeneous Neumann boundary data at $x=0$ and $x=\ell$;
\begin{align}
\label{defR1}
  R^1:f\in L^2(0,\ell)\longmapsto u\in H^2(0,\ell)\quad \text{solving}\quad
  \left\{
  \begin{array}{l}
     \big[1 - \frac{h_0^2}{3} \partial_x^2 \big]u = f \\  \partial_xu(0) = \partial_x u(\ell) = 0.
  \end{array}\right.\ .
\end{align}
An important observation is the following commutation property,
\begin{equation}\label{R0R1comm}
\forall f \in L^2(0,\ell), \qquad R_0\partial_x f=\partial_x R_1 f.
\end{equation}

\subsubsection{Well-posedness of the initial boundary value problem}

We prove here that the initial boundary value problem formed by \eqref{eq:BA} with boundary conditions 
\begin{equation}\label{BCeq}
\zeta(\cdot,0)=g_0 \quad\mbox{ and }\quad  \zeta(\cdot,\ell)=g_\ell,
\end{equation}
and initial condition 
\begin{equation}\label{ICeq}
(\zeta,q)(0,\cdot)=(\zeta^{\rm in},q^{\rm in})
\end{equation}
 is well posed.  Remarking that from the first equation of \eqref{eq:BP} one has $\frac{{\rm d}}{{\rm d} t}(\zeta(t,0))=-\partial_x q(t,0)$, and proceeding similarly at $x=\ell$,  a necessary compatibility condition on the initial and boundary data to allow the possibility of a solution which is of class $C^1$ at the origin is that
\begin{equation}\label{CC0}
\begin{cases}
\zeta^{\rm in}(0)=g_0(0),\\
-\partial_x q^{\rm in}(0)=\dot{g}_0(0),
\end{cases}
\quad \mbox{ and }\quad
\begin{cases}
 \zeta^{\rm in}(\ell)=g_\ell(0),\\
-\partial_x q^{\rm in}(\ell)=\dot{g}_\ell(0).
\end{cases}
\end{equation}
 
 The key step, in the spirit of \cite{Lannes_Weynans_2020,beck2023numerical} is to reformulate the problem as an initial value problem (no boundary condition).
\begin{proposition}\label{propreformflat}
Assume that the initial and boundary data $(\zeta^{\rm in},q^{\rm in})$ and $(g_0,g_\ell)$ satisfy the compatibility condition \eqref{CC0}.  Then the two following assertions are equivalent:\\
{\bf i.} The couple $(\zeta,q)$ is a regular solution to \eqref{eq:BA} such that the depth $h$ never vanishes and with boundary conditions \eqref{BCeq} and initial condition \eqref{ICeq}.\\
{\bf ii.} The couple $(\zeta,q)$ is a regular solution such that the depth $h$ never vanishes to 
\begin{equation}
\begin{cases}
\partial_t \zeta+\partial_x q=0,\\
\partial_t q +\partial_x (R_1f_{\rm NSW})=\dot{q}_0{\mathfrak s}^{(0)}+\dot{q}_\ell{\mathfrak s}^{(\ell)},
\end{cases}
\quad \text{in}\ (0,\ell),
\label{eq:BA-reformulated}
\end{equation}
with initial condition \eqref{ICeq}, and where $q_0$ and $q_\ell$ solve the ODE
\begin{equation}\label{eq:BA-boundarylayer}
{\mathfrak S}'
\begin{pmatrix} \dot{q}_0 \\ \dot{q}_\ell \end{pmatrix}+
  \frac{3}{h_0^2}
\begin{pmatrix}
 f_{\rm NSW} (g_0,q_0)\\
f_{\rm NSW} (g_\ell,q_\ell)
 \end{pmatrix}
 =
 -
 \begin{pmatrix}
 \ddot g_0 \\
 \ddot g_\ell
 \end{pmatrix}
 +
  \frac{3}{h_0^2}
\begin{pmatrix}
(R_1 f_{\rm NSW})_{\vert_{x=0}}     \\
(R_1 f_{\rm NSW})_{\vert_{x=\ell}}      
 \end{pmatrix},
\end{equation}
with ${\mathfrak S}'$ defined in \eqref{defSprime}, and with initial condition 
\begin{equation}\label{ICeq2}
q_0(0)=q^{\rm in}(0) \quad\mbox{ and }\quad q_\ell(0)=q^{\rm in}(\ell).
\end{equation}
 \end{proposition}
\begin{proof}
Let  $(\zeta,q)$ be a regular solution to  \eqref{eq:BA} with boundary conditions $\zeta(\cdot,0)=g_0$ and $\zeta(\cdot,\ell)=g_\ell$ and initial condition $(\zeta,q)(0,\cdot)=(\zeta^{\rm in},q^{\rm in})$. By definition of $R_0$, the second equation of \eqref{eq:BA} can be rewritten as
$$
\partial_t q + R_0\partial_x f_{\rm NSW}=\dot{q}_0{\mathfrak s}^{(0)}+\dot{q}_\ell{\mathfrak s}^{(\ell)},
$$
where $q_0$ and $q_\ell$ denote the trace of $q$ at $x=0$ and $x=\ell$ respectively.
Differentiating with respect to $x$ we get
$$
\partial_t \partial_x q + \partial_x R_0\partial_x f_{\rm NSW}=\dot{q}_0{\mathfrak s}'^{(0)}+\dot{q}_\ell{\mathfrak s}'^{(\ell)}.
$$
Using the first equation one can replace $\partial_t \partial_x q=-\partial^2_t\zeta$ while the identities $R_0\partial_x=\partial_x R_1$ and $  (1 - \frac{h_0^2}{3} \partial_x^2 ) R_1={\rm Id}$ allow us to deduce that
$$
-\partial_t^2\zeta  -\frac{3}{h_0^2}(1-R_1)   f_{\rm NSW}=\dot{q}_0{\mathfrak s}'_{(0)}+\dot{q}_\ell{\mathfrak s}'_{(\ell)}.
$$
Taking the trace of this equation at $x=0$ and $x=\ell$ respectively, we therefore obtain the following differential system on $q_0$ and $q_\ell$,
$$
\begin{cases}
{\mathfrak s}'_{(0)}(0) \dot{q}_0+{\mathfrak s}'_{(\ell)}(0) \dot{q}_\ell&=  -\frac{3}{h_0^2} f_{\rm NSW} (g_0,q_0)+ \frac{3}{h_0^2}(R_1 f_{\rm NSW})_{\vert_{x=0}}       - \ddot g_0 \\
{\mathfrak s}'_{(0)}(\ell) \dot{q}_0+{\mathfrak s}'_{(\ell)}(\ell) \dot{q}_\ell&= -\frac{3}{h_0^2} f_{\rm NSW} (g_\ell,q_\ell)+ \frac{3}{h_0^2}(R_1 f_{\rm NSW})_{\vert_{x=\ell}}      - \ddot g_\ell;
\end{cases}
$$
this shows that the first point of the proposition implies the second one. For the reverse implication, we can observe that if $(\zeta,q)$ and $(q_0,q_l)$ solve the above system then $(\zeta,q)$ solves the Boussinesq system \eqref{eq:BA}. By taking the trace of the second equation of \eqref{eq:BA-reformulated} at $x=0$ and $x=\ell$ respectively, we also obtain that $\frac{{\rm d}}{{\rm d}t} q(t,0)=\dot{q}_0(t)$ and $\frac{{\rm d}}{{\rm d}t} q(t,\ell)=\dot{q}_\ell(t)$. Since moreover we take by assumption $q(0,0)=q_0(0)$ and $q(0,\ell)=q_\ell(0)$, we deduce that $q(t,0)=q_0(t)$ and $q(t,\ell)=q_\ell(t)$ for all times. In contrast,  the
boundary conditions \eqref{BCeq} are not necessarily satisfied. However, as for the derivation of \eqref{eq:BA-boundarylayer}, we obtain that $(\zeta_0(t),\zeta_\ell(t)):=(\zeta(t,0),\zeta(t,\ell))$ satisfies
$$
{\mathfrak S}'
\begin{pmatrix} \dot{q}_0 \\ \dot{q}_\ell \end{pmatrix}+
  \frac{3}{h_0^2}
\begin{pmatrix}
 f_{\rm NSW} (\zeta_0,q_0)\\
f_{\rm NSW} (\zeta_\ell,q_\ell)
 \end{pmatrix}
 =
 -
 \begin{pmatrix}
 \ddot \zeta_0 \\
 \ddot \zeta_\ell
 \end{pmatrix}
 +
  \frac{3}{h_0^2}
\begin{pmatrix}
(R_1 f_{\rm NSW})_{\vert_{x=0}}     \\
(R_1 f_{\rm NSW})_{\vert_{x=\ell}}      
 \end{pmatrix},
$$
and it follows that $(\zeta_0,\zeta_\ell)$ and $(g_0,g_\ell)$ satisfy the same second order ODE. We can therefore conclude that $(\zeta_0,\zeta_\ell)=(g_0,g_\ell)$ if and only if
\[
(\zeta_0(0),\zeta_\ell(0))=(g_0(0),g_\ell(0))
\quad\mbox{ and }\quad
(\dot{\zeta}_0(0),\dot{\zeta}_\ell(0))=(\dot{g}_0(0),\dot{g}_\ell(0)).\]
Replacing $\partial_t\zeta$ by $-\partial_x q$, these conditions are equivalent to \eqref{CC0}. It follows that the boundary conditions \eqref{BCeq} are also satisfied. This concludes the proof of the proposition.
\end{proof}

The initial value problem formed by \eqref{eq:BA-reformulated} and \eqref{eq:BA-boundarylayer} with initial conditions \eqref{ICeq} and \eqref{ICeq2} is actually an ordinary differential equation on $H^{n}(0,\ell)\times H^{n+1}(0,\ell)\times {\mathbb R}^2$ for $(\zeta,q,q_0,q_\ell)$ if $n$ is a nonnegative integer. The existence of a local in time solution therefore follows immediately from Cauchy-Lipschitz theorem (we do not give details for the proof, which can easily be adapted from \cite{Lannes_Weynans_2020}). Note that the condition $\inf (h_0+\zeta^{\rm in})>0$ ensures that the shallow water flux $f_{\rm NSW}$ is well defined. By Proposition \ref{propreformflat}, this solution furnishes a solution to the original initial boundary value problem \eqref{eq:BA}, \eqref{BCeq} and  \eqref{ICeq}.
\begin{proposition}\label{propWPBA}
Let $g_0,g_\ell \in C^\infty({\mathbb R}^+)$. Let also $n\in {\mathbb N}\backslash\{0\}$ and $(\zeta^{\rm in},q^{\rm in})\in H^{n}(0,\ell)\times H^{n+1}(0,\ell)$ be such that $\inf (h_0+\zeta^{\rm in})>0$ and satisfying the compatibility conditions \eqref{CC0}.\\
Then there exists a maximal existence time $T^*>0$ and a unique solution $(\zeta,q,q_0,q_\ell)\in C^\infty([0,T^*); H^{n}(0,\ell)\times H^{n+1}(0,\ell)\times {\mathbb R}^2)$ to  \eqref{eq:BA-reformulated} and \eqref{eq:BA-boundarylayer} with initial conditions \eqref{ICeq} and \eqref{ICeq2}, and for all $t\in [0,T^*)$, one has $q_0(t)=q(t,0)$ and $q_\ell(t)=q(t,\ell)$.
\end{proposition}

\subsection{The Boussinesq-Abbott model with topography}
\label{subsec:boussinesq-peregrine}

We now consider the full Boussinesq-Abbott model which contains additional topography terms compared to \eqref{eq:BA}. Throughout this section, we assume that the bottom is parameterized by a function $-h_0+b$, with $b$ a smooth function on $[0,\ell]$. We assume also that the depth of the fluid at rest never vanishes, namely,
\begin{equation}\label{condhb}
\inf_{(0,\ell)} (h_0-b) >0.
\end{equation}
Denoting $h_b=h_0-b$, the Boussinesq-Abbott model is given by
\begin{align}\label{eq:BP}
  \left\{
  \begin{array}{l}
    \partial_t \zeta + \partial_x q = 0 \\
    (1  + h_b{\mathcal T}_b)\partial_t q + \partial_x f_{\text{NSW}} = -{\mathtt g}h\partial_x b
  \end{array}\right.\quad \text{in}\ (0,\ell)\ ,
\end{align}
where $h$ now contains an additional topography term, $h=h_0+\zeta- b$, and where the shallow water flux is still given by 
$$
 f_{\text{NSW}}=f_{\rm NSW}(\zeta,q)=\frac{1}{2} {\mathtt g} (h^2-h_0^2) +\frac{1}{h}q^2,
 $$
  which is the same expression as in the case of a flat topography, except that $h$ now depends also on $b$.
The second order operator  $\mathcal T_b$ is given by
\begin{align}\label{eq:operator-Tb}
{  \mathcal T}_b(\cdot) = -\frac{1}{3 h_b}\partial_x\Big(h_b^3\partial_x\frac{(\cdot)}{h_b}\Big)
 + \frac{1}{2}\partial_x^2 b
\end{align}
and  can be alternatively written under the form (see Section 5.6 of \cite{Lannes_book})
$$
{\mathcal T}_b(\cdot)=S^*\big(h_b S(\cdot)\big)+\frac{1}{4}\frac{(\partial_x b)^2}{h_b}
$$
with
\begin{equation}\label{defS}
S(\cdot)=-\frac{1}{\sqrt{3}}h_b\partial_x(\frac{1}{h_b}\cdot)+\frac{\sqrt{3}}{2}\frac{\partial_x b}{h_b}.
\end{equation}

\subsubsection{Notations and preliminary results}

We generalize here, in the presence of a non-flat topography, the concepts introduced in  Section~\ref{subsec:reformulation-BL}. We introduce the operator $R_{b}^0$ as the inverse of $(1 + h_b{\mathcal T}_b)$ with homogeneous Dirichlet boundary conditions, that is to say
\begin{align}\label{defR0b}
  R^0_b:f\in L^2(0,\ell)\longmapsto u\in H^2(0,\ell)\quad \text{solving}\quad
  \left\{
  \begin{array}{l}
    \big[1 + h_b {\mathcal T}_b\big] u = f \\ u(0) = u(\ell) = 0
  \end{array}\right. ;
\end{align}
here again, the operator $R^0_b$ can be also extended as an operator mapping $H^{-1}(0,\ell)$ to $H^1(0,\ell)$.
We also introduce the solutions  ${\mathfrak s}^{(b,0)}$ and  ${\mathfrak s}^{(b,\ell)}$ of the homogeneous equation with non homogeneous boundary conditions,
\begin{equation}\label{defs0lb}
  \left\{
  \begin{array}{l}
    \big[1 + h_b {\mathcal T}_b\big] {\mathfrak s}^{(b,0)} = 0 \\ {\mathfrak s}^{(b,0)}(0) = 1, \qquad {\mathfrak s}^{(b,0)}(\ell) = 0
  \end{array}\right.
\quad\mbox{ and }\quad
 \left\{
  \begin{array}{l}
    \big[1 + h_b {\mathcal T}_b\big]  {\mathfrak s}^{(b,\ell)} = 0 \\ {\mathfrak s}^{(b,\ell)}(0) = 0, \qquad {\mathfrak s}^{(b,\ell)}(\ell) = 1
  \end{array}\right. 
\end{equation}
and we also denote by ${\mathfrak S}_{b}'$ the matrix
\begin{equation}\label{defSbprime}
{\mathfrak S}_b'=\begin{pmatrix}
({\mathfrak s}^{(b,0)})'(0)  & ({\mathfrak s}^{(b,\ell)})'(0) \\
({\mathfrak s}^{(b,0)})'(\ell)  & ({\mathfrak s}^{(b,\ell)})'(\ell) 
\end{pmatrix}.
\end{equation}
Note that contrary to the case of a flat topography, there is in general no explicit expression for ${\mathfrak s}_{(b,0)}$ and  ${\mathfrak s}_{(b,\ell)}$. We can however prove that ${\mathfrak S}_b'$ is invertible under an assumption which is satisfied by all the topography profiles we considered in this paper, as well as by typical beach profiles.
\begin{proposition}\label{propS'}
Assume that \eqref{condhb} is satisfied and that in addition
\begin{equation}\label{condhb2}
\min_{(0,\ell)}(1+\frac{1}{4} (\partial_x b)^2+\frac{1}{6}h_b \partial_x^2 b)>0.
\end{equation}
Then  the matrix  ${\mathfrak S}_b'$ is invertible and $({\mathfrak s}^{(b,0)})'(0)\neq 0$ and $({\mathfrak s}^{(b,\ell)})'(\ell)\neq 0$.
\end{proposition}
\begin{proof}
We can remark that
$$
1+h_b{\mathcal T}_b=-\frac{1}{3}\partial_x (h_b^2 \partial_x \cdot )-\frac{1}{3} h_b (\partial_x b) \partial_x + (1+\frac{1}{3} (\partial_x b)^2+\frac{1}{6}h_b \partial_x^2 b).
$$
For all $u\in H^1(0,\ell)$ such that $\partial_x u(0)=\partial_x u(\ell)=0$, we obtain therefore that 
$$
\int_0^\ell u (1+h_b{\mathcal T}_b) u=\frac{1}{3} \int_0^\ell (h_b^2 (\partial_x u)^2 - h_b(\partial_x b) u \partial_x u)
+\int_0^\ell (1+\frac{1}{3} (\partial_x b)^2+\frac{1}{6}h_b \partial_x^2 b)u^2.
$$
Since for all $\epsilon >0$, one has
$$
-h_b (\partial_x b) u \partial_x u\geq -\frac{\epsilon}{2} h_b^2 (\partial_x u)^2 -\frac{1}{2\epsilon}(\partial_x b)^2 u^2,
$$
we deduce that
$$
\int_0^\ell u (1+h_b{\mathcal T}_b) u  \geq  \frac{1}{3}(1-\frac{\epsilon}{2})  \int_0^\ell h_b^2 (\partial_x u)^2
+\int_0^\ell (1+\frac{1}{3}(1-\frac{1}{2\epsilon}) (\partial_x b)^2+\frac{1}{6}h_b \partial_x^2 b)u^2.
$$
Under the assumptions of the proposition, it is possible to find $\frac{1}{2}<\epsilon<2$ and a constant $C_\epsilon>0$ such that
$$
\int_0^\ell u (1+h_b{\mathcal T}_b) u  \geq C_\epsilon \vert u \vert_{H^1(0,\ell)}^2;
$$
in particular, if $u$ solves $(1+h_b{\mathcal T}_b) u = 0$ and satisfies  $\partial_x u(0)=\partial_x u(\ell)=0$, then one has $u\equiv 0$.\\
Now, if the matrix ${\mathfrak S}_b'$ were not invertible, there would exist $(\lambda,\mu)\neq (0,0)$ such that the function ${\mathfrak s}:= \lambda ({\mathfrak s}^{(b,0)})+\mu ({\mathfrak s}^{(b,\ell)})$ would satisfy ${\mathfrak s}'(0)={\mathfrak s}'(\ell)=0$. From the above considerations, one would have ${\mathfrak s}\equiv 0$, which is not possible since ${\mathfrak s}(0)=\lambda$ and  ${\mathfrak s}(\ell)=\mu$ and $(\lambda,\mu)\neq (0,0)$.\\
In order to prove that $({\mathfrak s}^{(b,0)})'(0)\neq 0$, we can modify the arguments above to get that if $u\in H^1(0,\ell)$ is such that $u'(0)=0$ and $u(\ell)=0$ then $u\equiv 0$. If we had $({\mathfrak s}^{(b,0)})'(0)=0$ then $({\mathfrak s}^{(b,0)})$ would be such a function and would identically vanish, which is absurd. We get similarly that $({\mathfrak s}^{(b,\ell)})'(\ell)\neq 0$, which concludes the proof.
\end{proof}

Another important difference with the case of flat topography is that the commutation property \eqref{R0R1comm} is no longer true, that is, in general we have  $R^0_b\partial_x \neq R^1_b\partial_x$  if we define $R^1_b$ as the inverse of  $(1 + h_b{\mathcal T}_b)$ with homogeneous Neumann boundary conditions. We can however define $R^1_b$ as follows
\begin{align*}
  R^1_b:f\in L^2(0,\ell)\longmapsto u\in H^2(0,\ell)
\end{align*}
 where $u$ solves
 \begin{align}\label{defR1b}
  \left\{
  \begin{array}{l}
    \big[1+ \frac{1}{h_b}(h_b S)\big( \frac{h_b}{\alpha_b }(h_b S)^* \big) \big] u = f ,\\ (h_b S)^*u(0) = (h_b S)^*u(\ell) = 0,
  \end{array}\right.
\end{align}
and with $\alpha_b=1+\frac{1}{4}(\partial_x b)^2$. The following proposition shows that there is a commutation property that generalizes the identity $R^0\partial_x=\partial_x R^1$ for the Boussinesq-Abbott model with topography. 
\begin{proposition}\label{propcommutgen}
For all $f\in H^1(0,\ell)$, the following identity holds
$$
R^0_b\partial_x f=  \sqrt{3}\frac{h_b}{\alpha_b }(h_b S)^* (R^1_b (\frac{1}{h_b^2}f))
-\frac{3}{2}R^0_b \big(\frac{\partial_x b}{h_b} f\big),
$$
where $\alpha_b=1+\frac{1}{4}(\partial_x b)^2$, and where we recall that $S(\cdot)$ is defined in \eqref{defS}.
\end{proposition}
\begin{remark}\label{rem:BP-generalization}
When $b\equiv 0$ then $h_b=h_0$, $R^0_b=R^0$, $R^1_b=R^1$ and $S^*=\frac{1}{\sqrt{3}}\partial_x$, so that the identity of the lemma coincides with the identity $R^0\partial_x=\partial_x R^1$ obtained previously.
\end{remark}
\begin{proof}
Let us first remark that one can write
\begin{equation}\label{R0equiv}
1 + h_b {\mathcal T}_b=\alpha_b+ h_b S^*\big(h_b S(\cdot)\big);
\end{equation}
it follows that the equation $ (1 + h_b {\mathcal T}_b) u = \partial_x f$ can be equivalently written under the form
\begin{align}
\nonumber
\big[\frac{\alpha_b}{h_b}+S^*\big(h_b S(\cdot)\big) \big] u&=\frac{1}{h_b}\partial_x f \\
\label{decompR0}
&=S^* (\frac{\sqrt{3}}{h_b}f)-\frac{3}{2}\frac{1}{h_b^2}(\partial_x b) f.
\end{align}
We now need the following lemma.
\begin{lemma}
For all $\widetilde{f}\in L^2(0,\ell)$, the following identity holds,
$$
R^0_b \big(h_b S^* \widetilde{f}\big)=  \frac{h_b}{\alpha_b }(h_b S)^* (R^1_b (\frac{\widetilde{f}}{h_b})).
$$
\end{lemma}
\begin{proof}[Proof of the lemma]
Let us write $v=R^1_b(\frac{\widetilde{f}}{h_b})$; by definition, one has
$$
  \left\{
  \begin{array}{l}
    \big[1+ \frac{1}{h_b}(h_b S)\big( \frac{h_b}{\alpha_b }(h_b S)^*(\cdot) \big) \big] v = \frac{1}{h_b}\widetilde{f} \\ (h_b S)^*v(0) = (h_b S)^*v(\ell) = 0;
  \end{array}\right.
  $$
  Applying $(h_bS)^*$ to the equation, one finds that
  $$
   \big[ (h_b S)^*+ S^*(h_b S)\big( \frac{h_b}{\alpha_b }(h_b S)^*(\cdot) \big) \big] v =S^* \widetilde{f} 
  $$
  or equivalently
  $$
     \big[ \alpha_b + h_bS^*(h_b S) \big] \big( \frac{h_b}{\alpha_b }(h_b S)^*(v)\big) =h_bS^* \widetilde{f} ;
  $$
  using \eqref{R0equiv}, we deduce that
  $$
    \big[ 1+h_b{\mathcal T}_b \big] \big( \frac{h_b}{\alpha_b }(h_b S)^*(v)\big) =h_b S^* \widetilde{f} 
  $$
  Since moreover $\big( \frac{h_b}{\alpha_b }(h_b S)^*(v)\big) $ vanishes at $x=0,\ell$, we use the definition of $R_b^0$ to conclude that
  $$
   \frac{h_b}{\alpha_b }(h_b S)^*(v) =R^0_b \big(h_b S^* \widetilde{f}\big),
  $$
  which proves the lemma. 
\end{proof}
Owing to \eqref{R0equiv}, one can write \eqref{decompR0} under the form
$$
R^0_b\partial_x f= R^0_b\big( h_b S^* (\frac{\sqrt{3}}{h_b}f) \big)-\frac{3}{2}R^0_b \big(\frac{\partial_x b}{h_b} f\big),
$$
so that the result follows by applying the lemma with $\widetilde{f}=\frac{\sqrt{3}}{h_b}f$.
\end{proof}

\subsubsection{Well-posedness of the initial boundary value problem}\label{subsect:BPWB}

We prove here that the initial boundary value problem formed by \eqref{eq:BP} with boundary conditions 
\begin{equation}\label{BCeqb}
\zeta(\cdot,0)=g_0 \quad\mbox{ and }\quad  \zeta(\cdot,\ell)=g_\ell,
\end{equation}
and initial condition 
\begin{equation}\label{ICeqb}
(\zeta,q)(0,\cdot)=(\zeta^{\rm in},q^{\rm in})
\end{equation}
is well posed. As seen previously, a necessary compatibility condition on the initial and boundary data to allow the possibility of a solution which is of class $C^1$ at the origin is that
\begin{equation}\label{CC}
\begin{cases}
\zeta^{\rm in}(0)=g_0(0),\\
-\partial_x q^{\rm in}(0)=\dot{g}_0(0),
\end{cases}
\quad \mbox{ and }\quad
\begin{cases}
 \zeta^{\rm in}(\ell)=g_\ell(0),\\
-\partial_x q^{\rm in}(\ell)=\dot{g}_\ell(0).
\end{cases}
\end{equation}
We first state the following generalization of Proposition \ref{propreformflat} to the Boussinesq-Abbott model, in which we use the notation
\begin{equation}\label{defgb}
B(\zeta,q)=-{\mathtt g} h \partial_x b+\frac{3}{2}\frac{\partial_x b}{h_b}f_{\rm NSW}.
\end{equation}
We also introduce the two dimensional vectors
\begin{equation}\label{defVbdry}
V_{\rm bdry}(\zeta_0,q_0,\zeta_\ell,q_\ell)=\begin{pmatrix}
 \frac{3}{h_b(0)^2}f_{\rm NSW} (\zeta_0,q_0) \\
\frac{3}{h_b(\ell)^2}f_{\rm NSW} (\zeta_\ell,q_\ell) 
 \end{pmatrix}
 \end{equation}
 which depends only on the boundary values of $(\zeta,q)$, and
 \begin{equation}\label{defVint}
 V_{\rm int}[\zeta,q]=
\begin{pmatrix}
3(R_b^1 (\frac{1}{h_b^2}f_{\rm NSW}) )_{\vert_{x=0}}    - \big(\partial_x (R^0_b B) \big)_{\vert_{x=0}}  \\
3(R_b^1 (\frac{1}{h_b^2}f_{\rm NSW}) )_{\vert_{x=\ell}}   -\big( \partial_x (R^0_b B)\big)_{\vert_{x=\ell}}
 \end{pmatrix},
\end{equation}
which is a nonlocal function of the interior values of $\zeta$ and $q$ on the whole interval $(0,\ell)$.
\begin{proposition}\label{propreformgen}
Let the bottom parametrization $b$ satisfy \eqref{condhb} and \eqref{condhb2}. Assume that the initial and boundary data $(\zeta^{\rm in},q^{\rm in})$  and $(g_0, g_\ell)$ satisfy the compatibility condition \eqref{CC}. Then the two following assertions are equivalent:\\
{\bf i.} The couple $(\zeta,q)$ is a regular solution to \eqref{eq:BP} such that the depth $h$ never vanishes and with boundary conditions \eqref{BCeqb} and initial condition \eqref{ICeqb}.\\
{\bf ii.} The couple $(\zeta,q)$ is a regular solution such that the depth $h$ never vanishes to 
\begin{equation}
\begin{cases}
\partial_t \zeta+\partial_x q=0,\\
\partial_t q + \sqrt{3}\frac{h_b}{\alpha_b }(h_b S)^* (R^1_b (\frac{1}{h_b^2}f_{\rm NSW}))
= R^0_b B + \dot{q}_0{\mathfrak s}^{(b,0)}+\dot{q}_\ell{\mathfrak s}^{(b,\ell)},
\end{cases}
\quad \text{in}\ (0,\ell),
\label{eq:BP-reformulated}
\end{equation}
where  we recall that $(h_b S)^*= \frac{1}{\sqrt{3}}\frac{1}{h_b} \partial_x(h_b^2\cdot ) + \frac{\sqrt{3}}{2}\partial_x b$ and that $B(\zeta,q)$ is given by~\eqref{defgb},
and with initial condition~\eqref{ICeqb}, while $q_0$ and $q_\ell$ solve the ODE
\begin{align}
&
{\mathfrak S}_b'
\begin{pmatrix} \dot{q}_0 \\ \dot{q}_\ell \end{pmatrix}+
V_{\rm bdry}(g_0,q_0,g_\ell,q_\ell)
 =V_{\rm int}[\zeta,q]
   -
 \begin{pmatrix}
 \ddot g_0 \\
 \ddot g_\ell
 \end{pmatrix},
  \label{eq:BP-boundarylayer}
\end{align}
with ${\mathfrak S}'_b$ defined in \eqref{defSbprime}, and with initial condition 
\begin{equation}\label{ICeq2b}
q_0(0)=q^{\rm in}(0) \quad \mbox{ and }\quad q_\ell(0)=q^{\rm in}(\ell).
\end{equation}
\end{proposition}
\begin{remark}
When $b=0$ (flat topography), one can check that ${\mathfrak f}_b=R^1f_{\rm NSW}$ and ${\mathfrak g}_b=0$, so that \eqref{eq:BP-reformulated} coincides as expected with \eqref{eq:BA-reformulated}.
\end{remark}
\begin{proof}
We proceed as in the proof of Proposition \ref{propreformflat}. Applying $R^0_b$ to the equation in $q$, and using Proposition \ref{propcommutgen}, we first obtain
\begin{align*}
\partial_t q + \sqrt{3}\frac{h_b}{\alpha_b }(h_b S)^* (R^1_b (\frac{1}{h_b^2}f_{\rm NSW}))
&-\frac{3}{2}R^0_b \big(\frac{\partial_x b}{h_b} f_{\rm NSW}\big)\\*
&= - {\mathtt g} R_b^0(h\partial_x b) + \dot{q}_0{\mathfrak s}^{(b,0)}+\dot{q}_\ell{\mathfrak s}^{(b,\ell)}.
\end{align*}
Regrouping the terms involving $R_b^0$, one deduces the second equation of \eqref{eq:BP-reformulated}.
Applying $\partial_x$ to this equation and using the equation on $\zeta$, we then get
\begin{align}
\nonumber
-\partial_t^2\zeta + \sqrt{3} \partial_x\big[ \frac{h_b}{\alpha_b }(h_b S)^* (R^1_b (\frac{1}{h_b^2}f_{\rm NSW}))\big]
-\frac{3}{2}\partial_x\big[ R^0_b \big(\frac{\partial_x b}{h_b} f_{\rm NSW}\big) \big]\\
\label{eqSev}
= - {\mathtt g} \partial_x R_b^0(h\partial_x b) + \dot{q}_0({\mathfrak s}^{(b,0)})'+\dot{q}_\ell({\mathfrak s}^{(b,\ell)})'.
\end{align}
We now need the following lemma.
\begin{lemma}\label{lem:relation-Rb1}
For all $\widetilde{f}\in L^2(0,\ell)$, one has
$$
\sqrt{3} \partial_x\big[ \frac{h_b}{\alpha_b}(h_b S)^* (R^1_b \widetilde{f})\big]_{\vert_{x=0,\ell}}
= - 3\big[ ( \widetilde{f}- R^1_b\widetilde{f})\big]_{\vert_{x=0,\ell}}.
$$
\end{lemma}
\begin{proof}
One can notice that $\sqrt{3}\partial_x=-3 S+\beta$, where $\beta$ is a function whose exact expression is of no importance here. We then have
$$
\sqrt{3} \partial_x\big[ \frac{h_b}{\alpha_b }(h_b S)^* (R^1_b \widetilde{f})\big]=-\frac{3}{h_b} (h_b S)\big[ \frac{h_b}{\alpha_b }(h_b S)^* (R^1_b \widetilde{f})\big]+\beta \frac{h_b}{\alpha_b }(h_b S)^* (R^1_b \widetilde{f});
$$
by definition of $R_b^1$, one has $(h_b S)\big[ \frac{h_b}{\alpha_b }(h_b S)^* (R^1_b \widetilde{f})\big]=h_b \widetilde{f}-h_b R^1_b\widetilde{f}$ so that
$$
\sqrt{3} \partial_x\big[ \frac{h_b}{\alpha_b }(h_b S)^* (R^1_b \widetilde{f})\big]=-3 ( \widetilde{f}- R^1_b\widetilde{f})+\beta \frac{h_b}{\alpha_b }(h_b S)^* (R^1_b \widetilde{f}).
$$
Since by construction $(h_b S)^*(R^1_b \widetilde{f})$ vanishes at $x=0$ and $x=\ell$, the result follows upon taking the trace of the above identity at the two boundary points.
\end{proof}
Taking the trace of \eqref{eqSev} at $x=0$ ad $x=\ell$ and using the lemma with $\widetilde{f}=\frac{1}{h_b^2}f_{\rm NSW}$, we get that $q_0$ and $q_l$ satisfy the differential system
$$
\begin{cases}
\dot{q}_0{\mathfrak s}'_{(b,0)}(0)+\dot{q}_\ell{\mathfrak s}'_{(b,\ell)}(0) - {\mathtt g} \big(\partial_x R_b^0(h\partial_x b)\big)_0 = \\
\qquad - \ddot{g}_0 
-\frac{3}{h_b(0)^2}f_{\rm NSW}(g_0,q_0)
+3 \big(R^1_b(\frac{1}{h_b^2} f_{\rm NSW})\big)_0
-\frac{3}{2}\big[ \partial_xR^0_b \big(\frac{\partial_x b}{h_b} f_{\rm NSW}\big) \big]_0,\\
\dot{q}_\ell{\mathfrak s}'_{(b,0)}(\ell)+\dot{q}_\ell{\mathfrak s}'_{(b,\ell)}(\ell) - {\mathtt g} \big(\partial_x R_b^0(h\partial_x b)\big)_\ell = \\
\qquad -\ddot{g}_\ell
-\frac{3}{h_b(\ell)^2}f_{\rm NSW}(g_\ell,q_\ell)
+3 \big(R^1_b(\frac{1}{h_b^2} f_{\rm NSW})\big)_\ell
-\frac{3}{2}\big[ \partial_x R^0_b \big(\frac{\partial_x b}{h_b} f_{\rm NSW}\big) \big]_\ell,
\end{cases}
$$
which corresponds to~\eqref{eq:BP-boundarylayer}. The end of the proof is as in Proposition~\ref{propreformflat}.
\end{proof}
The well-posedness of the initial value problem given in the second point of the proposition is simply obtained as for Proposition \ref{propWPBA} from Cauchy-Lipschitz theorem.
\begin{proposition}\label{propWPBP}
Assume that $b$ is a smooth function satisfying~\eqref{condhb} and~\eqref{condhb2}. Let $g_0,g_\ell \in C^\infty({\mathbb R}^+)$. Let also $n\in {\mathbb N}\backslash\{0\}$ and $(\zeta^{\rm in},q^{\rm in})\in H^{n}(0,\ell)\times H^{n+1}(0,\ell)$ be such that $\inf_{[0,\ell]} (h_0-b+\zeta^{\rm in})>0$ and  assume that the compatibility conditions~\eqref{CC} hold.
Then there exists a maximal existence time $T^*>0$ and a unique solution $(\zeta,q,q_0,q_\ell)\in C^\infty([0,T^*); H^{n}(0,\ell)\times H^{n+1}(0,\ell)\times {\mathbb R}^2)$ to  \eqref{eq:BP-reformulated} and \eqref{eq:BP-boundarylayer} with initial conditions \eqref{ICeqb} and \eqref{ICeq2b}, and for all $t\in [0,T^*)$, one has $q_0(t)=q(t,0)$ and $q_\ell(t)=q(t,\ell)$, and $\inf_{[0,\ell]} (h_0-b+\zeta(t,\cdot))>0$.
\end{proposition}
\begin{proof}
The system \eqref{eq:BP-reformulated}-\eqref{eq:BP-boundarylayer} can be rewritten as 
$$
\frac{{\rm d}}{{\rm d} t} {\mathbb U}={\mathbb F}[t,{\mathbb U}]
$$
with ${\mathbb U}=(\zeta,q, q_0,q_\ell)^{\rm T}$ and 
$$
{\mathbb F}[t,{\mathbb U}]=(-\partial_x q,-\partial_x ({\mathfrak f}_b[\zeta,q])+{\mathfrak g}_b[\zeta,q]+\begin{psmallmatrix} {\mathfrak s}^{(b,0)} \\ {\mathfrak s}^{(b,\ell)}\end{psmallmatrix}\cdot \boldsymbol{\mathfrak h}_b[t,{\mathbb U}],
 \boldsymbol{\mathfrak h}_b[t,{\mathbb U}]^{\rm T})^{\rm T},
$$
where
$$
\boldsymbol{\mathfrak h}_b[t,{\mathbb U}]=({\mathfrak S}_b')^{-1}\big[ -V_{\rm bdry}(g_0(t),q_0,g_\ell(t),q_\ell)+V_{\rm int}[\zeta,q]-\begin{pmatrix} \ddot{g}_0(t) \\ \ddot{g}_\ell(t) \end{pmatrix} \big].
$$
Let ${\mathbb X}=H^n(0,\ell)\times H^{n+1}(0,\ell)\times {\mathbb R}^2$ and $\Omega\subset {\mathbb X}$ the open subset of ${\mathbb X}$ consisting of all ${\mathbb U}=(\zeta,q, q_0,q_\ell)^{\rm T}\in {\mathbb X}$ such that $\inf_{(0,\ell)} (h_0+\zeta-b)>0$. Then it follows from standard product estimates, and because $R^0_b$ and $R^1_b$ map $H^n(0,\ell)$ to $H^{n+2}(0,\ell)$, that ${\mathbb F} : {\mathbb R}^+\times \Omega \to {\mathbb X}$ is well-defined, continuous and locally-Lipschitz with respect to ${\mathbb U}$. By Cauchy-Lipschitz theorem, for all initial data ${\mathbb U}^{\rm in}\in \Omega$ there exists therefore a maximal solution ${\mathbb U}\in C^1([0,T^*],\Omega)$, with $T^*>0$, and moreover this solution belongs to ${\mathbb U}\in C^\infty([0,T^*],\Omega)$.
The result then follows from Proposition~\ref{propreformgen}.
\end{proof}

\subsection{A reformulation adapted to well-balancedness}
\label{subsec:unify-models-wb}

Since the steady state $(\zeta,q)=(0,0)$ (lake at rest) solves the system \eqref{eq:BP-reformulated}-\eqref{eq:BP-boundarylayer} with boundary data $g_0=g_\ell=0$, one gets in particular the following identity from the second equation  of \eqref{eq:BP-reformulated}
$$
\sqrt{3}\frac{h_b}{\alpha_b}(h_b S)^* (R^1_b (\frac{1}{h_b^2}f_{\rm NSW}(0,0)))=R^0_b B(0,0).
$$
In the presence of topography, both terms of this identity are nonzero. Subtracting this identity to the second equation of  \eqref{eq:BP-reformulated}, one obtains the following equivalent formulation
\begin{equation}
\begin{cases}
\partial_t \zeta+\partial_x q=0,\\
\partial_t q + \sqrt{3}\frac{h_b}{\alpha_b }(h_b S)^* (R^1_b (\frac{1}{h_b^2} \widetilde{f}_{\rm NSW}))
= R^0_b \widetilde{B} + \dot{q}_0{\mathfrak s}^{(b,0)}+\dot{q}_\ell{\mathfrak s}^{(b,\ell)},
\end{cases}
\quad \text{in}\ (0,\ell),
\label{eq:BP-reformulated2}
\end{equation}
where, denoting $\widetilde{f}_{\rm NSW}(\zeta,q):=f_{\rm NSW}(\zeta,q)-f_{\rm NSW}(0,0)$, one has
\begin{equation}\label{defBtilde}
\widetilde{B}(\zeta,q)=-{\mathtt g} \zeta \partial_x b+\frac{3}{2}\frac{\partial_x b}{h_b}\widetilde{f}_{\rm NSW}.
\end{equation}
Proceeding similarly with \eqref{eq:BP-boundarylayer} we get that $q_0$ and $q_\ell$ solve the ODE 
\begin{align}
&
{\mathfrak S}_b'
\begin{pmatrix} \dot{q}_0 \\ \dot{q}_\ell \end{pmatrix}+
\widetilde{V}_{\rm bdry}(g_0,q_0,g_\ell,q_\ell)
 =\widetilde{V}_{\rm int}[\zeta,q]
   -
 \begin{pmatrix}
 \ddot g_0 \\
 \ddot g_\ell
 \end{pmatrix},
  \label{eq:BP-boundarylayer2}
\end{align}
with
\begin{equation}\label{defVbdrytilde}
\widetilde{V}_{\rm bdry}(\zeta_0,q_0,\zeta_\ell,q_\ell)=\begin{pmatrix}
 \frac{3}{h_b(0)^2}\widetilde{f}_{\rm NSW} (\zeta_0,q_0) \\
\frac{3}{h_b(\ell)^2}\widetilde{f}_{\rm NSW} (\zeta_\ell,q_\ell) 
 \end{pmatrix}
 \end{equation}
 and
 \begin{equation}\label{defVint_}
\widetilde{V}_{\rm int}[\zeta,q]=
\begin{pmatrix}
3(R_b^1 (\frac{1}{h_b^2}\widetilde{f}_{\rm NSW}) )_{\vert_{x=0}}    - \big(\partial_x (R^0_b\widetilde{B}) \big)_{\vert_{x=0}}  \\
3(R_b^1 (\frac{1}{h_b^2}\widetilde{f}_{\rm NSW}) )_{\vert_{x=\ell}}   -\big( \partial_x (R^0_b \widetilde{B})\big)_{\vert_{x=\ell}}
 \end{pmatrix},
\end{equation}
In order to obtain well-balanced numerical schemes, we will use \eqref{eq:BP-reformulated2} and \eqref{eq:BP-boundarylayer2} rather than \eqref{eq:BP-reformulated} and \eqref{eq:BP-boundarylayer}.

\section{More general boundary conditions}
\label{subsec:generic-bnd}

In the previous sections,  we considered the Boussinesq-Abbott system \eqref{eq:BP} on a finite interval $(0,\ell)$ and with imposed boundary conditions at $x=0$ and $x=\ell$ on the surface elevation $\zeta$. In some situations, one may have to impose rather a boundary condition on $q$, or an a function of $\zeta$ and $q$ such as a Riemann invariant of the nonlinear shallow water system. 

Denoting $(\zeta_0,q_0)=(\zeta,q)_{\vert_{x=0}}$ and $(\zeta_\ell,q_\ell)=(\zeta,q)_{\vert_{x=\ell}}$, we study in this section the possibility of imposing more general boundary conditions which take the form
\begin{align}\label{eq:generic-bnd}
  \xi^{+}_0(\zeta_0, q_0) = g_0,\qquad
  \xi^{-}_\ell(\zeta_\ell, q_\ell) = g_\ell,
\end{align}
where the data $(g_0,g_\ell)$ are given functions of time, and where $\xi^+_0(\zeta,q)$ and $\xi^-_\ell(\zeta,q)$ are functions of $\zeta$ and $q$; we refer to $\xi^+_0(\zeta,q)$ and $\xi^-_\ell(\zeta,q)$ as the \emph{input} functions.

In order to solve the initial boundary value problem associated with \eqref{eq:generic-bnd}, an essential point is to compute the traces of $\zeta$ and $q$ at the boundaries $x=0$ and $x=\ell$. This is done in two steps. We first need two {\it output functions} that can be found in terms of the input functions; we then reconstruct the traces of $\zeta$ and $q$ in terms of the input and output functions. These two steps are explained in~\S \ref{subsectoutput}. We then derive in~\S \ref{secteqout} an ODE that can be used to compute the output functions. In~\S \ref{sectcompat}, we discuss the issue of providing initial data to this ODE, as well as the issue of compatibility conditions between initial and boundary data. We can then state and prove in~\S \ref{sectWP} our main well-posedness result. We finally introduce in~\S \ref{sectAS} the issue of asymptotic stability which is an open theoretical problem that will be numerically investigated later in Section~\ref{sectnum}.

\subsection{The output functions and the reconstruction mappings}\label{subsectoutput}

 We assume that there exist two other functions $ \xi^{-}_0(\cdot, \cdot)$ and $ \xi^{+}_\ell(\cdot, \cdot)$, referred to as the \emph{output} functions, and such that $(\zeta_0,q_0)$ can be recovered from the knowledge of $g_0$ and $ \xi^{-}_0(\zeta_0, q_0)$ and similarly, $(\zeta_\ell,q_\ell)$ can be recovered from the knowledge of $g_\ell$ and $ \xi^{+}_0(\zeta_\ell, q_\ell)$. More precisely, we make the following assumption.
\begin{assumption}\label{assH}
There exists a non-empty connected open set ${\mathcal U}\subset {\mathbb R}^2$ verifying $(0,0)\in \mathcal U$ such that the following conditions hold:\\
{\bf i.} The input functions $\xi^+_0: {\mathcal U}\to {\mathbb R}$ and  $\xi^-_\ell: {\mathcal U}\to {\mathbb R}$ are well defined and smooth;\\
{\bf ii.} There are two output functions $\xi^-_0: {\mathcal U}\to {\mathbb R}$ and  $\xi^+_\ell: {\mathcal U}\to {\mathbb R}$ and two open sets ${\mathcal V}_0$ and ${\mathcal V}_\ell$ that contain the range of the mappings  $\xi_0: {\mathcal U}\ni (\zeta,q)\mapsto (\xi_0^+(\zeta,q),\xi_0^-(\zeta,q))\in {\mathbb R}^2$  and $\xi_\ell: {\mathcal U}\ni (\zeta,q)\mapsto (\xi_\ell^+(\zeta,q),\xi_\ell^-(\zeta,q))\in {\mathbb R}^2$ respectively, as well as two smooth mappings ${\mathcal H}_0: {\mathcal V}_0\to {\mathbb R}^2$ and ${\mathcal H}_\ell: {\mathcal V}_\ell \to {\mathbb R}^2$ such that for all $(\zeta,q)\in {\mathcal U}$, one has
\begin{equation}\label{eq:H-reconstruction}
{\mathcal H}_0(\xi^+_0(\zeta,q),\xi^-_0(\zeta,q))=(\zeta,q)^{\rm T}
\quad \mbox{ and }\quad
{\mathcal H}_\ell(\xi^+_\ell(\zeta,q),\xi^-_\ell(\zeta,q))=(\zeta,q)^{\rm T};
\end{equation}
we refer to ${\mathcal H}_0$ and ${\mathcal H}_\ell$ as the \emph{reconstruction} mappings.
\end{assumption}
\begin{example}\label{ex:xi-H}
For instance, let us impose the boundary conditions $\zeta_0=g_0$ and $R^-(\zeta_\ell,q_\ell)=g_\ell$, where $R^-(\zeta,q)=\frac{q}{h_0 + \zeta - b} - 2\sqrt{{\mathtt g}(h_0 + \zeta - b)}$ is the left-going Riemann invariant associated with the nonlinear shallow water equations. With the above notations, this amounts to take
$$
\xi_0^+(\zeta,q)=\zeta \quad\mbox{ and }\quad \xi_\ell^-(\zeta,q)=R^-(\zeta,q).
$$
We can then choose
$$
\xi_0^-(\zeta,q)=q \quad\mbox{ and }\quad \xi_\ell^+(\zeta,q)=R^+(\zeta,q),
$$
where  $R^+=\frac{q}{h_0 + \zeta - b} + 2\sqrt{{\mathtt g}(h_0 + \zeta - b)}$ is the right-going analogue of $R^-$. The corresponding maps ${\mathcal H}_0$ and   ${\mathcal H}_\ell$ are then given by
\begin{align*}
{\mathcal H}_0(\xi^+,\xi^-)&=(\xi^+,\xi^-)^{\rm T},\\
{\mathcal H}_\ell(\xi^+,\xi^-)&= \big(\frac{1}{16{\mathtt g}}(\xi^+ - \xi^-)^2 - h_0 + b, \frac{\xi^+ + \xi^-}{32{\mathtt g}} (\xi^+-\xi^-)^2 \big)^{\rm T}.
\end{align*}
\end{example}

\subsection{Equations for the output functions}\label{secteqout}

If Assumption \ref{assH} holds, the knowledge of the output functions $\xi_0^-(\zeta,q)$ and $\xi_\ell^+(\zeta,q)$, together with the boundary conditions set on  the input functions $\xi_0^+(\zeta,q)$ and $\xi_\ell^-(\zeta,q)$ allow one to determine the traces of $\zeta$ and $q$ at $x=0$ and $x=\ell$. The issue know is to be able to compute the output functions $\xi_0^-(\zeta,q)$ and $\xi_\ell^+(\zeta,q)$. In the (hyperbolic) case of the nonlinear shallow water equations, this can be done using the characteristic equations satisfied by the Riemann invariants, provided that the input functions $\xi^\pm_0$ and $\xi^\pm_\ell$ are correctly chosen (see~\cite{IguchiLannes} for a full analysis of $1D$ hyperbolic initial boundary value problems). 

Due to the presence of dispersion, there are no Riemann invariants associated with the Boussinesq-Abbott model. However, we can notice that the relation~\eqref{eq:BP-boundarylayer2} stems from a more general differential identity that relates $\zeta_0$, $q_0$, $\zeta_\ell$ and $q_\ell$, namely,
\begin{align}
&
{\mathfrak S}_b'
\begin{pmatrix} \dot{q}_0 \\ \dot{q}_\ell \end{pmatrix}+
\widetilde{V}_{\rm bdry}(\zeta_0,q_0,\zeta_\ell,q_\ell)
 =\widetilde{V}_{\rm int}[\zeta,q]
   -
 \begin{pmatrix}
 \ddot \zeta_0 \\
 \ddot \zeta_\ell
 \end{pmatrix} ;
  \label{eq:BP-boundarylayer3}
\end{align}
in fact we see that~\eqref{eq:BP-boundarylayer3} reduces to~\eqref{eq:BP-boundarylayer2} when enforcing $(\zeta_0,\zeta_\ell) = (g_0,g_\ell)$. We now show how to use this differential identity to compute $\underline{\xi}_0^-:=\xi_0^-({\zeta}_0,{q}_0)$ and $\underline{\xi}_\ell^+:=\xi_\ell^+(\zeta_\ell,q_\ell)$ when enforcing the general boundary conditions~\eqref{eq:generic-bnd}. By definition of ${\mathcal H}_0$ and ${\mathcal H}_\ell$, and using the boundary conditions \eqref{eq:generic-bnd}, one has
$$
\begin{pmatrix}
\zeta_0 \\ q_0 \end{pmatrix}={\mathcal H}_0(g_0,\underline{\xi}_0^-)
\quad \mbox{ and }\quad
\begin{pmatrix}
\zeta_\ell \\ q_\ell \end{pmatrix}={\mathcal H}_\ell(\underline{\xi}_\ell^+,g_\ell);
$$
denoting  ${\mathcal H}_0=({\mathcal H}_{0,1},{\mathcal H}_{0,2})^{\rm T}$ and  ${\mathcal H}_\ell=({\mathcal H}_{\ell,1},{\mathcal H}_{\ell,2})^{\rm T}$, one deduces that
\begin{equation}\label{dotq}
\dot{q}_0=\nabla {\mathcal H}_{0,2}(g_0,\underline{\xi}^-_0)\cdot \begin{pmatrix} \dot{g}_0 \\ \dot{\underline{\xi}}_0^- \end{pmatrix}
\quad\mbox{ and }\quad
\dot{q}_\ell=\nabla {\mathcal H}_{\ell,2}(\underline{\xi}_\ell^+,g_\ell)\cdot \begin{pmatrix} \dot{\underline{\xi}}_\ell^+ \\ \dot{g}_\ell \end{pmatrix}
\end{equation}
as well as
\begin{align*}
\ddot{\zeta}_0&=\nabla {\mathcal H}_{0,1}(g_0,\underline{\xi}^-_0)\cdot \begin{pmatrix} \ddot{g}_0 \\ \ddot{\underline{\xi}}_0^- \end{pmatrix}
+ \begin{pmatrix} \dot{g}_0 \\ \dot{\underline{\xi}}_0^- \end{pmatrix}\cdot {\rm Hess}_{{\mathcal H}_{0,1}}(g_0,\underline{\xi}^-_0) \begin{pmatrix} \dot{g}_0 \\ \dot{\underline{\xi}}_0^- \end{pmatrix}\\
\ddot{\zeta}_\ell&=\nabla {\mathcal H}_{\ell,1}(\underline{\xi}_\ell^+,g_\ell)\cdot \begin{pmatrix} \ddot{\underline{\xi}}_\ell^+ \\ \ddot{ g_\ell} \end{pmatrix}
+ \begin{pmatrix} \dot{\underline{\xi}}_\ell^+ \\ \dot{g}_\ell  \end{pmatrix}\cdot {\rm Hess}_{{\mathcal H}_{\ell,1}}(\underline{\xi}_\ell^+,g_\ell) \begin{pmatrix}  \dot{\underline{\xi}}_\ell^+ \\ \dot{g}_\ell \end{pmatrix},
\end{align*}
where ${\rm Hess}_{{\mathcal H}_{0,1}}$ and ${\rm Hess}_{{\mathcal H}_{\ell,1}}$  denote the $2\times 2$ Hessian matrices of ${\mathcal H}_{0,1}$ and  ${\mathcal H}_{\ell,1}$ respectively.
Introducing the matrices
$$
D_j=\begin{pmatrix} \partial_1{\mathcal H}_{0,j}(g_0,\underline{\xi}^-_0) & 0 \\ 0 & \partial_2 {\mathcal H}_{\ell,j}(\underline{\xi}_\ell^+,g_\ell) \end{pmatrix}
$$
and
$$
\widetilde{D}_j=\begin{pmatrix} \partial_2{\mathcal H}_{0,j}(g_0,\underline{\xi}^-_0) & 0 \\ 0 & \partial_1 {\mathcal H}_{\ell,j}(\underline{\xi}_\ell^+,g_\ell) \end{pmatrix},
$$
and the quadratic forms defined on ${\mathbb R}^2$ by
\begin{align*}
{\mathcal Q_1}(u_0,v_0)&= 
 \begin{pmatrix} u_0 \\ v_0 \end{pmatrix}\cdot {\rm Hess}_{{\mathcal H}_{0,1}}(g_0,\underline{\xi}^-_0) \begin{pmatrix} u_0 \\ v_0\end{pmatrix},\\
 {\mathcal Q_2}(u_\ell,v_\ell)&= 
 \begin{pmatrix} u_\ell \\ v_\ell \end{pmatrix}\cdot {\rm Hess}_{{\mathcal H}_{\ell,1}}(\underline{\xi}_\ell^+,g_\ell) \begin{pmatrix} u_\ell \\ v_\ell\end{pmatrix},
\end{align*}
we deduce from \eqref{eq:BP-boundarylayer3} the following systems of two scalar ODEs on $(\underline{\xi}_0^-,\underline{\xi}_\ell^+)$, 
\begin{align}
\nonumber
\widetilde{D}_1\begin{pmatrix}  \ddot{\underline{\xi}}^-_0 \\ \ddot{\underline{\xi}}^+_\ell \end{pmatrix}
+
{\mathfrak S}'_b\widetilde{D}_2\begin{pmatrix}  \dot{\underline{\xi}}^-_0 \\ \dot{\underline{\xi}}^+_\ell \end{pmatrix}
&+
\begin{pmatrix}
{\mathcal Q}_1(\dot{g}_0,\dot{\underline{\xi}}_0^-) \\
{\mathcal Q}_2(\dot{\underline{\xi}}_\ell^+,\dot{g}_\ell)
\end{pmatrix}
+
\widetilde{V}_{\rm bdry}\big( {\mathcal H}_0(g_0,\underline{\xi}_0^-),{\mathcal H}_\ell(\underline{\xi}_\ell^+,g_\ell)\big)  \\
\label{CBODE}
& =\widetilde{V}_{\rm int}[\zeta,q]
 -
 {D}_1\begin{pmatrix}  \ddot{g}_0 \\ \ddot{g}_\ell \end{pmatrix}
-
{\mathfrak S}'_b{D}_2\begin{pmatrix}  \dot{g}_0 \\ \dot{g}_\ell \end{pmatrix}.
\end{align}
The order of this system of ODEs depends on whether the coefficients of $\widetilde{D}_1$ vanish or not. We make the following assumption which ensures that \eqref{CBODE} can be put as a system of explicit ODEs of first or second order on $\underline{\xi}_0^-$ and $\underline{\xi}_\ell^+$.
\begin{assumption}\label{assH2}
Under Assumption \ref{assH} and with the same notations, we assume that:\\
{\bf i.} Either $\partial_2{\mathcal H}_{0,1}\equiv0$ and $\partial_2{\mathcal H}_{0,2}$ does not vanish on ${\mathcal V}_0$, or $\partial_2{\mathcal H}_{0,1}$ does not vanish on ${\mathcal V}_0$;\\
{\bf ii.} Either $\partial_1{\mathcal H}_{\ell,1}\equiv0$ and $\partial_1{\mathcal H}_{\ell,2}$ does not vanish on ${\mathcal V}_0$, or $\partial_1{\mathcal H}_{\ell,1}$ does not vanish on ${\mathcal V}_\ell$
\end{assumption}
\begin{remark}\label{remvan}
If $\partial_2{\mathcal H}_{0,1}\equiv0$ on ${\mathcal V}_0$ then ${\mathcal Q}_1(\dot{g}_0,\dot{\underline{\xi}_0^-})=\partial_1^2{\mathcal H}_{0,1}(g_0,\underline{\xi}_0^-)\dot{g}_0^2$ which is independent of $\underline{\xi}_0^-$. The system of ODEs \eqref{CBODE} is therefore of order $1$ in $\underline{\xi}_0^-$ and can be put in explicit form if $\partial_2{\mathcal H}_{0,2}$ does not vanish. A similar comment can be made for the ODE on $\underline{\xi}_\ell^+$.
\end{remark}
\begin{example}\label{ex:xi-H2}
Considering the same configuration as in Example \ref{ex:xi-H}, and with the same notations, one readily checks that $\partial_2{\mathcal H}_{0,1}=0$ and $\partial_2{\mathcal H}_{0,2}=1$ so that the first point of the assumption is satisfied. Moreover, in this case,  ${\mathcal Q}_1=0$. We also compute $\partial_1{\mathcal H}_{\ell,1}=\frac{1}{8{\mathtt g}} (\xi^+-\xi^-)$. It is possible to choose the open set ${\mathcal V}_\ell$ in such a way that $\xi^+-\xi^->0$ for all $(\xi^+,\xi^-)\in {\mathcal V}_\ell$ provided that we assume the total water height $h$ never vanishes. Assumption \ref{assH2} is therefore satisfied. Recalling the definition \eqref{defSbprime} of ${\mathfrak S}'_b$, the system \eqref{CBODE} can be put under the form
\begin{align}
\nonumber
T
\begin{pmatrix} \dot{\underline{\xi}}_0^- \vspace{1mm}\\ \ddot{\underline{\xi}}_\ell^+ \end{pmatrix}
+
{\mathfrak S}'_b\widetilde{D}_2\begin{pmatrix}  0 \vspace{1mm}  \\ \dot{\xi}^+_\ell \end{pmatrix}
&+
\begin{pmatrix}
0 \vspace{1mm}\\
{\mathcal Q}_2(\dot{\underline{\xi}}_\ell^+,\dot{g}_\ell)
\end{pmatrix}
+
\widetilde{V}_{\rm bdry}\big( {\mathcal H}_0(g_0,\underline{\xi}_0^-),{\mathcal H}_\ell(\underline{\xi}_\ell^+,g_\ell)\big)   \\
\label{CBODE2}
& =\widetilde{V}_{\rm int}[\zeta,q]
 -
 {D}_1\begin{pmatrix}  \ddot{g}_0 \\ \ddot{g}_\ell \end{pmatrix}
-
{\mathfrak S}'_b{D}_2\begin{pmatrix}  \dot{g}_0 \\ \dot{g}_\ell \end{pmatrix},
\end{align}
where $T$ is the triangular matrix
$$
T=
\begin{pmatrix}
({\mathfrak s}^{b,0})'(0) & 0 \\ ({\mathfrak s}^{b,0})'(\ell) & \partial_1{\mathcal H}_{\ell,1}(\underline{\xi}_\ell^+,g_\ell)
\end{pmatrix} .
$$
Under the assumption on $b$ made in the statement of Proposition \ref{propS'}, one has $({\mathfrak s}^{b,0})'(0)\neq 0$ and $T$ is invertible, and the system \eqref{CBODE2} therefore furnishes an ODE in explicit form of first order in $\underline{\xi}_0^-$ and second order on $\underline{\xi}_\ell^+$. {\it For all the applications considered in this paper, we will always have ODEs of second order, except when the boundary condition is imposed on the surface elevation (as in this example at $x=0$). }
\end{example}

We can notice that if $\xi_0^+(\zeta,q)=q$ (boundary condition on the discharge $q$ at $x=0$), and if we choose $\xi_0^-(\zeta,q)=\zeta$ as output function, then ${\mathcal H}_{0}(\xi^+,\xi^-)=(\xi^-,\xi^+)$ and in particular $\partial_1{\mathcal H}_{0,1}\equiv 0$. We make the following assumption to ensure that if we make another choice of input function, then  $\partial_1{\mathcal H}_{0,1}$ does not vanish.
\begin{assumption}\label{assH3}
Under Assumption \ref{assH} and with the same notations, we assume that:\\
{\bf i.} If $\partial_1{\mathcal H}_{0,1}$ vanishes on ${\mathcal V}_0$ then $\xi_0^+(\zeta,q)=\xi_0^+(q)$ on ${\mathcal U}$. \\
{\bf ii.} If $\partial_2{\mathcal H}_{\ell,1}$ vanishes on ${\mathcal V}_\ell$ then $\xi_\ell^-(\zeta,q)=\xi_\ell^-(q)$ on ${\mathcal U}$.
\end{assumption}

\subsection{Compatibility conditions and initial conditions for the output functions}
\label{sectcompat}

 If the initial condition $(\zeta,q)_{\vert_{t=0}}=(\zeta^{\rm in},q^{\rm in})$ is imposed for the Boussinesq-Abbott system, then a necessary compatibility condition with the boundary conditions~\eqref{eq:generic-bnd} to allow solutions that are continuous at $(t=0,x=0,\ell)$ is that
 \begin{equation}\label{compB}
 \xi_0^+(\zeta^{\rm in}(0),q^{\rm in}(0))=g_0(0)
 \quad\mbox{ and }\quad
  \xi_\ell^-(\zeta^{\rm in}(\ell),q^{\rm in}(\ell))=g_\ell(0);
 \end{equation}
 these conditions generalize the first condition of \eqref{CC}.
 In order to solve \eqref{CBODE}, one also has to prescribe initial data on $\underline{\xi}^-_0$ and $\underline{\xi}_\ell^+$; one naturally takes
\begin{equation}\label{CIODE1}
\underline{\xi}^-_0(0)=\xi^-_0(\zeta^{\rm in}(0),q^{\rm in}(0))
\quad \mbox{ and }\quad 
\underline{\xi}^-_\ell(0)=\xi^+_\ell(\zeta^{\rm in}(\ell),q^{\rm in}(\ell)).
\end{equation}
By definition of ${\mathcal H}_0$ and ${\mathcal H}_\ell$, we know that
\begin{align*}
\dot{\zeta}_0&=\partial_1{\mathcal H}_{0,1}(g_0,\underline{\xi}_0^-)\dot{g}_0+\partial_2{\mathcal H}_{0,1}(g_0,\underline{\xi}_0^-)\dot{\underline{\xi}}_0^-,\\
\dot{\zeta}_\ell&=\partial_1{\mathcal H}_{\ell,1}(\underline{\xi}_\ell^+,g_\ell)\dot{\underline{\xi}}_\ell^++\partial_2{\mathcal H}_{\ell,1}(\underline{\xi}_\ell^+,g_\ell)\dot{g}_\ell.
\end{align*}
Using the first equation of the Boussinesq-Abbott system \eqref{eq:BP}, we can replace $\dot{\zeta}_0(t)$ by $-\partial_x q(t,0)$ (and proceed similarly at $x=\ell$) to obtain
\begin{equation}\label{CIODE2}
\begin{cases}
\partial_2{\mathcal H}_{0,1} (g_0(0),\underline{\xi}_0^-(0))  \, \dot{\underline{\xi}}_0^-&=  -\partial_x q^{\rm in}(0)-\partial_1{\mathcal H}_{0,1}(g_0(0),\underline{\xi}_0^-(0))\dot{g}_0(0),\\
\partial_1{\mathcal H}_{\ell,1} (\underline{\xi}_\ell^+(0),g_\ell(0)) \, \dot{\underline{\xi}}_\ell^+&= -\partial_x q^{\rm in}(\ell)-\partial_2{\mathcal H}_{\ell,1}(\underline{\xi}_\ell^+(0),g_\ell(0))\dot{g}_\ell(0),
\end{cases}
\end{equation}
where we assumed that the solution was regular enough to take the trace at $(t=0,x=0,\ell)$.
Depending on the situation, these conditions can be either a compatibility condition in the boundary and initial data to allow the possibility of regular solutions, or an initial condition for $\dot{\underline{\xi}}_0^-(0) $ or $\dot{\underline{\xi}}_\ell^+(0)$. More precisely:
\begin{itemize}[leftmargin=1.75em]
\item If  \eqref{CBODE} is of second order on $\underline{\xi}^-_0$ (resp. on $\underline{\xi}_\ell^+$), then an initial condition is also needed on  $\dot{\underline{\xi}}^-_0$ (resp. on $\dot{\underline{\xi}}_\ell^+$). 
Remarking further that under Assumption~\ref{assH2}, $\partial_2{\mathcal H}_{0,1}$ does not vanish if \eqref{CBODE} is of second order in $\underline{\xi}_0^-$ (resp. $\partial_1{\mathcal H}_{\ell,1}$ does not vanish if \eqref{CBODE} is of second order in $\underline{\xi}_\ell^+$), 
\eqref{CIODE2} furnishes the needed initial data for $\dot{\underline{\xi}}^-_0$ (resp. on $\dot{\underline{\xi}}_\ell^+$).
\item If  \eqref{CBODE} is of first order on $\underline{\xi}^-_0$ (resp. on $\underline{\xi}_\ell^+$) then by Assumption \ref{assH2} one has $\partial_2{\mathcal H}_{0,1}\equiv 0$ (resp. $\partial_1{\mathcal H}_{\ell,1}\equiv 0$), and \eqref{CIODE2} is then a compatibility condition on the data that generalizes the second equations in the compatibility conditions~\eqref{CC} imposed in Proposition \ref{propWPBP}.
\end{itemize}
\begin{example}
For the configuration considered in Examples \ref{ex:xi-H} and \ref{ex:xi-H2}, and assuming that \eqref{compB} holds, the ODE is of first order on  $\underline{\xi}^-_0$ and of second order in $\underline{\xi}_\ell^+$. The initial conditions we need to impose are therefore \eqref{CIODE1} and also, at $x=\ell$, the second condition of \eqref{CIODE2}. In order to expect $C^1$ solutions at the origin, we also need to impose at $x=0$ a compatibility condition on the initial and boundary data, which is given by the first equation of \eqref{CIODE1} (and which coincides in this case with the compatibility condition \eqref{CC}, namely, $-\partial_x q^{\rm in}(0)=\dot{g}_0$).
\end{example}

\subsection{Well-posedness of the initial boundary value problem with general boundary conditions}
\label{sectWP}

We can now state the main result of this paper, which proves the well-posedness of the Boussinesq-Abbott system \eqref{eq:BP} with general boundary conditions  \eqref{eq:generic-bnd}. Proposition \ref{propWPBP} is a particular case of the theorem, corresponding to $\xi_0^+(\zeta,q)=\xi_\ell^-(\zeta,q)=\zeta$.
\begin{theorem}\label{theoWP}
Assume that $b$ is a smooth function satisfying~\eqref{condhb} and~\eqref{condhb2}. Let also $\xi_0^+$ and $\xi_\ell^-$ be two input functions, $\xi_0^-$ and $\xi_\ell^+$ be two output functions, and ${\mathcal H}_0$ and ${\mathcal H}_\ell$ be reconstruction mappings that satisfy Assumptions~\ref{assH},~\ref{assH2} and~\ref{assH3}.\\
Let $g_0,g_\ell \in C^\infty({\mathbb R}^+)$. Let also $n\in {\mathbb N}\backslash\{0\}$ and $(\zeta^{\rm in},q^{\rm in})\in H^{n}(0,\ell)\times H^{n+1}(0,\ell)$ be such that $\inf_{[0,\ell]} (h_0-b+\zeta^{\rm in})>0$ and  assume that the compatibility condition~\eqref{compB} holds.
If moreover $\partial_2{\mathcal H}_{0,1}\equiv 0$ (resp. $\partial_1{\mathcal H}_{\ell,1}\equiv 0$) then we also assume that the first (resp. the second) compatibility condition of~\eqref{CIODE2} holds.\\
Then there exists a maximal existence time $T^*>0$ and a unique solution $(\zeta,q)\in C^\infty([0,T^*); H^{n}(0,\ell)\times H^{n+1}(0,\ell))$ to  the Boussinesq-Abbott system~\eqref{eq:BP}  with initial conditions~\eqref{ICeqb} and boundary conditions~\eqref{eq:generic-bnd}, and moreover for all $t\in (0,T^*)$, one has $\inf_{[0,\ell]} (h_0-b+\zeta(t,\cdot))>0$.
\end{theorem}
\begin{proof}
From the analysis of \S \ref{secteqout} and with the same notations, we know that if such a solution exists, then $(\zeta,q)$  solves the system~\eqref{eq:BP-reformulated2}. Substituting for $\dot{q}_0$ and $\dot{q}_\ell$ in the right-hand side of~\eqref{eq:BP-reformulated2} using \eqref{dotq}, we obtain that 
\begin{equation}\label{systref2}
\begin{cases}
\partial_t \zeta+\partial_x q=0,\\
\partial_t q + \sqrt{3}\frac{h_b}{\alpha_b }(h_b S)^* (R^1_b (\frac{1}{h_b^2} \widetilde{f}_{\rm NSW}))
= R^0_b \widetilde{B} \\
\hspace{1.5cm}+ \nabla {\mathcal H}_{0,2}(g_0,\underline{\xi}^-_0)\cdot \begin{pmatrix} \dot{g}_0 \\ \dot{\underline{\xi}}_0^- \end{pmatrix}
{\mathfrak s}^{(b,0)}+\nabla {\mathcal H}_{\ell,2}(\underline{\xi}_\ell^+,g_\ell)\cdot \begin{pmatrix} \dot{\underline{\xi}}_\ell^+ \\ \dot{g}_\ell \end{pmatrix}
{\mathfrak s}^{(b,\ell)},
\end{cases}
\end{equation}
This system is complemented by the ODE \eqref{CBODE} for $(\underline{\xi}_0^-,\underline{\xi}_\ell^+)$ which can be of first or second order in  $\underline{\xi}_0^-$ or $\underline{\xi}_\ell^+$. We now distinguish several cases and, for the sake of clarity, focus our attention at $x=0$, the adaptations to $x=\ell$ being straightforward.
\begin{itemize}[leftmargin=1.75em]
\item If $\partial_2{\mathcal H}_{0,1}\equiv 0$, then by Assumption \ref{assH2} we know that $\partial_2{\mathcal H}_{0,2}$ does not vanish and the ODE \eqref{CBODE} is explicit and of  first order in $\underline{\xi}_0^-$. The needed initial condition on  $\underline{\xi}_0^-$ is furnished by~\eqref{CIODE1}.
\item If this is not the case then $\partial_2{\mathcal H}_{0,1}$ does not vanish in virtue of Assumption~\ref{assH2}, and the ODE \eqref{CBODE} is explicit and of second order in $\underline{\xi}_0^-$. In addition to the  initial condition on  $\underline{\xi}_0^-$ furnished by~\eqref{CIODE1}, we need an initial condition on  $\dot{\underline{\xi}}_0^-$, which is furnished by \eqref{CIODE2}.
\end{itemize}
In all cases, we can use as in the proof of Proposition \ref{propWPBP} the Cauchy-Lipschitz theorem to construct a solution $(\zeta,q,\underline{\xi}_0^-,\underline{\xi}_\ell^+)$ to the initial value problem formed by~\eqref{systref2},~\eqref{CBODE} and the initial conditions. We need to prove now that $(\zeta,q)$ solves the Boussinesq-Abbott system~\eqref{eq:BP}  with initial conditions \eqref{ICeqb} and boundary conditions \eqref{eq:generic-bnd}. \\
Applying $(1+h_b{\mathcal T}_b)$ to~\eqref{systref2}, we obtain that $(\zeta,q)$ solves~\eqref{eq:BP}; it remains to prove that the boundary conditions~\eqref{eq:generic-bnd} are satisfied. 
Taking the trace of the second equation of~\eqref{systref2} at $x=0$ and $x=\ell$, we deduce from~\eqref{defR0b},~\eqref{defs0lb} and \eqref{defR1b} that
$$
\frac{{\rm d}}{{\rm d}t }q(t,0)=\frac{{\rm d}}{{\rm d}t } {\mathcal H}_{0,2}(g_0,\underline{\xi}^-_0)
\quad\mbox{ and }\quad
\frac{{\rm d}}{{\rm d}t }q(t,\ell)=\frac{{\rm d}}{{\rm d}t } {\mathcal H}_{\ell,2}(\underline{\xi}^+_\ell,g_\ell).
$$
Since we also know from Assumption \ref{assH} and \eqref{CIODE1} that $ {\mathcal H}_{0,2}(g_0,\underline{\xi}^-_0)_{\vert_{t=0}}=q^{\rm in}(0)$ and $ {\mathcal H}_{\ell,2}(\underline{\xi}^+_\ell,g_\ell)_{\vert_{t=0}}=q^{\rm in}(\ell)$, it follows that for all time, one has 
$q(t,0)= {\mathcal H}_{0,2}(g_0,\underline{\xi}^-_0)$ and $q(t,\ell)={\mathcal H}_{\ell,2}(\underline{\xi}^+_\ell,g_\ell)$. We here again have to distinguish two cases, and as above we  focus our attention at $x=0$.
\begin{itemize}[leftmargin=1.75em]
\item If $\xi_0^+(\zeta,q)=q$ then one has ${\mathcal H}_{0,2}(\xi^+,\xi^-)=\xi^+$  so that we deduce that $q(t,0)=g_0$, which is exactly the boundary condition  \eqref{eq:generic-bnd}.
\item If this is not the case, then we know by Assumption \ref{assH3} that $\partial_1{\mathcal H}_{0,1}(\xi^+,\xi^-)$ does not vanish on ${\mathcal V}_0$; therefore the ODE
\begin{align*}
\nonumber
\widetilde{D}_1\begin{pmatrix}  \ddot{\underline{\xi}}^-_0 \\ \ddot{\underline{\xi}}^+_\ell \end{pmatrix}
+
{\mathfrak S}'_b\widetilde{D}_2\begin{pmatrix}  \dot{\underline{\xi}}^-_0 \\ \dot{\underline{\xi}}^+_\ell \end{pmatrix}
&+
\begin{pmatrix}
{\mathcal Q}_1(\dot{\underline{\xi}}^+_0,\dot{\underline{\xi}}_0^-) \\
{\mathcal Q}_2(\dot{\underline{\xi}}_\ell^+,\dot{\underline{\xi}}^-_\ell)
\end{pmatrix}
+
\widetilde{V}_{\rm bdry}( {\mathcal H}_0(\underline{\xi}^+_0,\underline{\xi}_0^-),{\mathcal H}_\ell(\underline{\xi}_\ell^+,\underline{\xi}^+_\ell))  \\
& =\widetilde{V}_{\rm int}[\zeta,q]
 -
 {D}_1\begin{pmatrix}  \ddot{\underline{\xi}}^+_0 \\ \ddot{\underline{\xi}}^-_\ell \end{pmatrix}
-
{\mathfrak S}'_b{D}_2\begin{pmatrix}  \dot{\underline{\xi}}^+_0 \\ \dot{\underline{\xi}}^-_\ell \end{pmatrix}.
\end{align*}
which is obtained as  \eqref{CBODE}, is of second order in $\underline{\xi}_0^+=\xi^+_0(\zeta_0,q_0)$. It follows that $(\underline{\xi}_0^+,\underline{\xi}_\ell^-)$ satisfy the same ODE as $(g_0,g_\ell)$. These two quantities are therefore equal (which implies that the boundary condition \eqref{eq:generic-bnd} is satisfied) if $(\underline{\xi}_0^+(0),\underline{\xi}_\ell^-(0))=(g_0(0),g_\ell(0))$ and $(\dot{\underline{\xi}}_0^+(0),\dot{\underline{\xi}}_\ell^-(0))=(\dot{g}_0(0),\dot{g}_\ell(0))$. The first of these two conditions corresponds to \eqref{compB}. For the second one (focusing our attention on the case $x=0$), we observe by time differentiating the relation ${\mathcal H}_{0,1}(\underline{\xi}_0^+,\underline{\xi}_0^-)=\zeta(t,0)$ and using the first equation of \eqref{eq:BP} that
$$
\partial_1{\mathcal H}_{0,1}(\underline{\xi}_0^+(0),\underline{\xi}_0^-(0))\dot{\underline{\xi}}_0^+(0)+\partial_2{\mathcal H}_{0,1}(\underline{\xi}_0^+(0),\underline{\xi}_0^-(0))\dot{\underline{\xi}}_0^-(0)=-\partial_x q^{\rm in}(0).
$$
Since $\underline{\xi}_0^+(0)=g(0)$ by \eqref{compB}, we can use \eqref{CIODE2} (which holds by assumption if $\partial_2{\mathcal H}_{0,1}\equiv 0$ and by the initial condition imposed on $\dot{\underline{\xi}}_0^-$ otherwise) to obtain
$$
\partial_1{\mathcal H}_{0,1}(\underline{\xi}_0^+(0),\underline{\xi}_0^-(0)) \big(\dot{\underline{\xi}}_0^+(0)-\dot{g}_0(0)\big)=0,
$$
and since $\partial_1{\mathcal H}_{0,1}$ does not vanish, this implies that $\dot{\underline{\xi}}_0^+(0)=\dot{g}_0(0)$, so that the boundary condition  \eqref{eq:generic-bnd} is satisfied.
\end{itemize}
The proof of the theorem is then complete.
\end{proof}

\subsection{Asymptotic stability}
\label{sectAS}

For some applications, we have some information on $\zeta$ and $q$ at $x=0$ and $x=\ell$; this provides us with boundary conditions of the form \eqref{eq:generic-bnd} on some input functions $\xi_0^+$ and $\xi_\ell^-$.
 For instance, we may know the surface elevation $\zeta(t,0)$ and $\zeta(t,\ell)$ through buoys located at $x=0$ and $x=\ell$. By Theorem~\ref{theoWP} we are able to compute the solution $(\zeta,q)$ of the Boussinesq-Abbott equations~\eqref{eq:BP} on the full domain $(0,\ell)$ for some time interval $(0,T^*)$ with $T^*>0$ provided that we know the initial data $(\zeta^{\rm in},q^{\rm in})$. Unfortunately, for most applications, the initial data are not known (it is very complicated to measure the surface elevation and the horizontal discharge on the whole interval $(0,\ell)$). A question of high practical relevance is therefore the following: if we consider the solution $(\widetilde{\zeta},\widetilde{q})$ of the initial boundary value problem with the same boundary conditions but with initial data $ (\widetilde{\zeta}^{\rm in},\widetilde{q}^{\rm in})\neq (\zeta^{\rm in},q^{\rm in})$ do we have $(\widetilde{\zeta},\widetilde{q})\sim (\zeta,q)$ for large times? If this is the case, we shall say that the solution $(\zeta,q)$ of the initial boundary value problem 
formed by~\eqref{eq:BP}  with initial conditions~\eqref{ICeqb} and boundary conditions~\eqref{eq:generic-bnd} is \emph{asymptotically stable}.

Asymptotic stability cannot be expected in general, even if we consider only the linear equation. Consider for instance the linearized Boussinesq equations with a flat topography
\begin{equation}\label{BPlin}
\begin{cases}
\partial_t \zeta+\partial_x q=0,\\
(1-\frac{h_0^2}{3}\partial_x^2)\partial_t q+{\mathtt g}h_0\partial_x \zeta=0,
\end{cases}
\end{equation}
with boundary conditions
\begin{equation}\label{CBlin}
\zeta(t,0)=\zeta(t,\ell)=0.
\end{equation}
The rest state $(\zeta,q)=(0,0)$ is obviously a solution to this problem, associated with a homogeneous initial condition. This following proposition shows that the rest state is not asymptotically stable. In the statement, we use the notation
\begin{equation}\label{eq:linear-energy}
E(t)=\frac{1}{2}{\mathtt g}\vert \widetilde{\zeta}(t)\vert^2_{L^2(0,\ell)}+ \frac{1}{2} \frac{1}{h_0}\vert \widetilde{q}(t)\vert^2_{L^2(0,\ell)}+\frac{h_0}{6} \vert \partial_x \widetilde{q}(t)\vert^2_{L^2(0,\ell)}
\end{equation}
\begin{proposition}\label{prop:linear-asymptotic-stability}
If $(\widetilde{\zeta},\widetilde{q})$ is a smooth solution to \eqref{BPlin} and \eqref{CBlin} then for all times, one has 
$
E(t)=E(0)
$.
In particular, if $(\widetilde{\zeta},\widetilde{q})_{\vert_{t=0}}\neq (0,0)$ then one cannot have $(\widetilde{\zeta},\widetilde{q})\to (0,0)$ in $L^2(0,\ell)\times H^1(0,\ell)$.
\end{proposition}
\begin{proof}
Multiplying \eqref{BPlin} by $({\mathtt g}\widetilde{\zeta},\frac{1}{h_0}\widetilde{q})$ and integrating by parts and using the boundary conditions \eqref{CBlin}, we obtain
$$
\frac{{\rm d}}{{\rm d}t} E- \frac{h_0}{3}[\partial_x\partial_t \widetilde{q} \widetilde{q}]_0^\ell=0. 
$$
Using the first equation of \eqref{BPlin}, one has $\partial_x\partial_t \widetilde{q}=-\partial^2_t \widetilde{\zeta}$ so that using \eqref{CBlin} we deduce that $[\partial_x\partial_t \widetilde{q} \widetilde{q}]_0^\ell=0$, and the result follows easily.
\end{proof}

The fact that we allow in  \eqref{eq:generic-bnd} very general boundary conditions is important because it may allow us to use boundary conditions for which the solution furnished by Theorem \ref{theoWP} is asymptotically stable (provided that it is well defined globally in time of course). This question of asymptotic stability is a difficult open problem for the Boussinesq-Abbott equations. Even in the case of the much more widely studied nonlinear shallow water equations, very little is known. The most relevant result is the fact that the rest state is asymptotically stable if the input functions $\xi_0^+$ and $\xi_\ell^-$ at $x=0$ and $x=\ell$ are respectively the right and left-going Riemann invariants $R^+$ and $R^-$ \cite{BastinCoron}.\\
In \S \ref{sectnumAS}, we investigate numerically the asymptotic stability of the solution of the initial boundary value problem associated with various types of boundary conditions.

\section{Numerical schemes} \label{sectschemes}
In this section we detail the numerical approximation of the Boussinesq-Abbott system with topography \eqref{eq:BP}, namely, 
\begin{align*}
  \left\{
  \begin{array}{l}
    \partial_t \zeta + \partial_x q = 0 \\
    (1  + h_b{\mathcal T}_b)\partial_t q + \partial_x f_{\text{NSW}} = -{\mathtt g}h\partial_x b
  \end{array}\right.\quad \text{in}\ (0,\ell)\ ,
\end{align*}
with the general boundary conditions \eqref{eq:generic-bnd}, that is, with the notations of Section~\ref{subsec:generic-bnd}, 
\begin{align*}
  \xi^{+}_0(\zeta_0, q_0) = g_0,\qquad
  \xi^{-}_\ell(\zeta_\ell, q_\ell) = g_\ell,
\end{align*}
and with initial condition
$$
(\zeta,q)=(\zeta^{\rm in},q^{\rm in}) \quad \mbox{ at }\quad t=0.
$$
Under the assumptions of Theorem \ref{theoWP}, we know that there is a unique solution to this initial boundary-value problem. We show in this section how to approximate numerically this solution. We actually solve the reformulation \eqref{systref2} of the equations 
that reads
\begin{equation}\label{systeref3}
\begin{cases}
\partial_t \zeta+\partial_x q=0,\\
\partial_t q + \sqrt{3}\frac{h_b}{\alpha_b }(h_b S)^* (R^1_b (\frac{1}{h_b^2} \widetilde{f}_{\rm NSW}))
= R^0_b \widetilde{B} \\
\hspace{1.5cm}+ \nabla {\mathcal H}_{0,2}(g_0,\underline{\xi}^-_0)\cdot \begin{pmatrix} \dot{g}_0 \\ \dot{\underline{\xi}}_0^- \end{pmatrix}
{\mathfrak s}^{(b,0)}+\nabla {\mathcal H}_{\ell,2}(\underline{\xi}_\ell^+,g_\ell)\cdot \begin{pmatrix} \dot{\underline{\xi}}_\ell^+ \\ \dot{g}_\ell \end{pmatrix}
{\mathfrak s}^{(b,\ell)},
\end{cases}
\end{equation}
where $(\underline{\xi}_0^+,\underline{\xi}_\ell^-)$ solves the system of ODEs \eqref{CBODE} that we rewrite here for the sake of clarity,
\begin{align}
\nonumber
\widetilde{D}_1\begin{pmatrix}  \ddot{\underline{\xi}}^-_0 \\ \ddot{\underline{\xi}}^+_\ell \end{pmatrix}
+
{\mathfrak S}'_b\widetilde{D}_2\begin{pmatrix}  \dot{\underline{\xi}}^-_0 \\ \dot{\underline{\xi}}^+_\ell \end{pmatrix}
&+
\begin{pmatrix}
{\mathcal Q}_1(\dot{g}_0,\dot{\underline{\xi}}_0^-) \\
{\mathcal Q}_2(\dot{\underline{\xi}}_\ell^+,\dot{g}_\ell)
\end{pmatrix}
+
\widetilde{V}_{\rm bdry}\big( {\mathcal H}_0(g_0,\underline{\xi}_0^-),{\mathcal H}_\ell(\underline{\xi}_\ell^+,g_\ell)\big)  \\
\label{CBODE-recall}
& =\widetilde{V}_{\rm int}[\zeta,q]
 -
 {D}_1\begin{pmatrix}  \ddot{g}_0 \\ \ddot{g}_\ell \end{pmatrix}
-
{\mathfrak S}'_b{D}_2\begin{pmatrix}  \dot{g}_0 \\ \dot{g}_\ell \end{pmatrix}.
\end{align}
As explained in Section \ref{subsec:unify-models-wb}, this reformulation is adapted to well-balancedness in the sense that it enables to derive numerical schemes which naturally preserve the hydrostatic equilibrium $(\zeta,q) = (0,0)$.

Our strategy consists in an hybrid approach mixing a finite volumes scheme for the interior equations~\eqref{systeref3}, and a finite difference discretization of the nonlocal operators $R_b^0,R_b^1$ defined in~\eqref{defR0b}-\eqref{defR1b} and of the system of ODEs~\eqref{CBODE-recall} that relates the evolution of the output functions at the boundaries to the boundary data.
We explain in Section \ref{subsec:discr-op} how to discretize the various operators involved in~\eqref{systeref3}; we then propose in Section \ref{subsec:1st-order-discr} a first order Lax-Friedrichs scheme, and a second order MacCormack scheme in Section \ref{subsec:2nd-order-discr}.

\medbreak
\noindent{\bf Notations.}
We introduce a few notations used in the following lines. We consider the grid points $x_i = (i-1)\Delta x$ for all $1\leq i\leq N$ with $(N-1)\Delta x = \ell$. A dual mesh is then obtained as the $N$ cells centered on the $(x_i)_{1\leq i\leq N}$ and delimited by their interfaces $x_{i\pm1/2} = x_i\pm\Delta x/2$. The use of this dual mesh is not mandatory, but it will allow us to deal with the boundary conditions~\eqref{eq:generic-bnd} and the dispersive boundary layer directly in the border cells which are centered on $x_1 = 0$ and $x_N = \ell$. In particular, this rids us from the need to use interpolation formulas whenever a quantity has to be evaluated at the boundary of the domain. We denote the discrete times $t^n = n\Delta t$, and we consider $U_i^n = (\zeta_i^n, q_i^n)^T$ an approximation of $\frac{1}{\Delta x}\int_{x_i-1/2}^{x_i+1/2} (\zeta,q)^T(t^n,s)\mathrm ds$, the average of the solution in the cell $i$ at time $t^n$. A similar notation is adopted for the bathymetry $b$.

\subsection{Discrete operators}
\label{subsec:discr-op}

In order to discretize $(h_b S)$, its adjoint $(h_b S)^*$ and the nonlocal operators $R_b^0,R_b^1$, we shall make use of the second order finite difference operator $\delta_x:\mathbb R^N\rightarrow \mathbb R^N$ defined such that for any $v\in\mathbb R^N$ and $1\leq i\leq N$ we have
\begin{equation}\label{eq:deltax}
\delta_x (v)_i = \left\{\begin{array}{ll}
\frac{1}{2\Delta x}(v_{i+1} - v_{i-1}) & \text{if}\ 2\leq i\leq N-1, \\[.5em]
\frac{1}{2\Delta x}(-3v_1 + 4v_2 - v_3) & \text{if}\ i = 1, \\[.5em]
\frac{1}{2\Delta x}(v_{N-2} - 4v_{N-1} + 3v_N) & \text{if}\ i = N.
\end{array}\right. 
\end{equation}
For any index $1\leq i\leq N$, we define the quantities
\[
\alpha_{b,i} = 1 + \frac{1}{4}(\delta_x(b)_i)^2 , \quad h_{b,i} = h_0 - b_i ,
\]
and whenever considering the discrete setting, $\alpha_b$ and $h_b$ will from now on refer to the vectors with corresponding coefficients $(\alpha_{b,i})_i$ and $(h_{b,i})_i$. It is then possible to approximate $(h_b S)$ by $(\underline{h_b S})$ as follows:
\begin{equation}
\forall v\in\mathbb R^N,\quad
    (\underline{h_b S})(v)_i = -\frac{h_{b,i}^2}{\sqrt{3}} \delta_x\Big(\frac{v}{h_b}\Big)_i + \frac{\sqrt{3}}{2} \delta_x(b)_i v_i ,
\end{equation}
which is consistent with the definition $\eqref{defS}$ of $S(\cdot)$. In a similar fashion we introduce $(\underline{h_b S})^*$ defined as
\begin{align}\label{eq:hbS-star-discrete}
    (\underline{h_b S})^*(f)_i ={} & \frac{\alpha_{b,i}}{\sqrt{3}h_{b,i}} \delta_x \Big(\frac{h_b^2 f}{\alpha_b}\Big)_i + \Big(\frac{h_{b,i}}{\sqrt{3}}\frac{\delta_x(\alpha_b)_i}{\alpha_{b,i}} + \frac{\sqrt{3}}{2}\delta_x b_i\Big) f_i ,
\end{align}
and that is shown to be consistent with $(h_b S)^*$ up to a $O(\Delta x^2)$ error.

Next we discretize the nonlocal operator $R_b^0$ by $\underline R_b^0$ defined implicitly as
\begin{align}
    \underline R_b^0:f&\in\mathbb R^N\longmapsto v = \underline R_b^0f \in \mathbb R^N \text{ such that} \nonumber\\
    & \begin{cases}
    \displaystyle v_i - \delta_x\Big(\frac{h_b^3}{3}\delta_x\Big(\frac{v}{h_b}\Big)\Big)_i + \frac{h_{b,i}}{2\Delta x^2}(b_{i+1} - 2b_i + b_{i-1}) v_i = f_i , & 1 < i < N, \\
    v_1 = 0, \quad v_N = 0.
    \end{cases}  \label{eq:Rtheta0-interior}
\end{align}
and one readily checks that this is a consistent approximation of $(1 + h_b\mathcal T_b)$ at order $2$, with the operator $\mathcal T_b$ given in~\eqref{eq:operator-Tb}. Likewise $R_b^1$ is discretized by $\underline R_b^1$ defined implicitly as
\begin{align}
    \underline R_b^1:f&\in\mathbb R^N\longmapsto v = \underline R_b^1f \in \mathbb R^N \text{ such that} \nonumber\\
    & \begin{cases}
    \displaystyle v_i + \frac{1}{h_{b,i}}(\underline{h_b S})\Big(\frac{h_{b}}{\alpha_{b}}(\underline{h_b S})^*(v)\Big)_i = f_i , & 1 < i < N, \\
    (\underline{h_b S})^*(v)_1 = 0, \quad (\underline{h_b S})^*(v)_N = 0,
    \end{cases}  \label{eq:Rtheta1-interior}
\end{align}
which is consistent with~\eqref{defR1b}. The definitions~\eqref{eq:Rtheta0-interior}-\eqref{eq:Rtheta1-interior} are implicit since the practical computation of $\underline R_b^0$ and $\underline R_b^1$ applied to some vector will require the resolution of a linear system, and this will be the costliest step of the algorithm. However, these discrete operators do not evolve in time and it is thus possible to assemble them and perform their factorization in a prepossessing step, which greatly reduces the computational time.

\begin{remark}\label{rem:hbS-kernel}
Owing to the homogeneous conditions from~\eqref{eq:Rtheta1-interior} verified in the border cells, one has $(\underline{h_b S})^*(f)_1 = (\underline{h_b S})^*(f)_N = 0$ for any vector $f$ in $\mathrm{Im}(\underline R_b^1)$. This is the discrete counterpart to the fact that $(h_b S)^*(R_b^1\cdot)_{\vert_{x=0,\ell}}\equiv 0$ in $L^2(0,\ell)$, which enables to obtain the relation~\eqref{CBODE-recall} --- ensuing from Lemma~\ref{lem:relation-Rb1} after applying $\partial_x$ to the second equation of \eqref{systeref3} and taking the trace at $x=0,\ell$ --- and hereby to ensure that the boundary conditions~\eqref{eq:generic-bnd} are satisfied.
Therefore it is important to have this property at the discrete level to remain consistent with the continuous reformulation~\eqref{systeref3}-\eqref{CBODE-recall}.
\end{remark}

\subsection{First order Lax-Friedrichs scheme}
\label{subsec:1st-order-discr}

We begin by describing how the system~\eqref{systeref3} for the interior values is approximated. For the first equation, the free surface elevation is updated in cells $2\leq i\leq N-1$ using a finite volumes strategy
\begin{equation}\label{eq:zeta-O1-update}
    \frac{\zeta_i^{n+1} - \zeta_i^n}{\Delta t} + \frac{1}{\Delta x}\Big(q_{i+1/2}^n - q_{i-1/2}^n \Big) = 0 ,
\end{equation}
where $q_{i\pm1/2}^n$ is the Lax-Friedrichs numerical flux defined by
\begin{equation}
q_{i+1/2}^n = \frac{1}{2}(q_{i}^n + q_{i+1}^n) - \frac{\Delta x}{2\Delta t}(\zeta_{i+1}^n - \zeta_i^n) .
\end{equation}
The boundary values $\zeta_1^{n+1},\zeta_N^{n+1}$ will be deduced upon approximating the ODE for the output functions.
To approximate the second equation of~\eqref{systeref3} we first introduce the vectors $\widetilde f_{\mathrm{NSW}}^n, \widetilde B^n\in\mathbb R^N$ with coefficients
\begin{gather}\label{eq:discr-flux-source}
    \begin{cases}
    \displaystyle
    \widetilde f_{\mathrm{NSW},i}^n = \frac{(q_i^n)^2}{h_0 + \zeta_i^n - b_i} + \frac{\mathtt g}{2}((\zeta_i^n)^2 + 2h_{b,i}\zeta_i^n) , \\
    \displaystyle
    \widetilde B_i^n = -{\mathtt g} \zeta_i^n \delta_x(b)_i + \frac{3}{2}\frac{\delta_x(b)_i}{h_{b,i}}\widetilde{f}_{\mathrm{NSW},i}^n ,
    \end{cases}
\end{gather}
which is consistent with the definition~\eqref{defBtilde} of $\widetilde B(\zeta,q)$. The update of the discharge in cells $1\leq i\leq N$ makes use of the discrete operators introduced in Section~\ref{subsec:discr-op} as follows:
\begin{align}\label{eq:q-O1-update}
&
    \frac{q_i^{n+1} - q_i^n}{\Delta t} + \sqrt{3}\frac{h_{b,i}}{\alpha_{b,i}} (\underline{h_b S})^*\Big(\underline R_b^1\Big(\frac{1}{h_b^2}\widetilde f_{\mathrm{NSW}}^n\Big)\Big)_i ={} \\
    & \quad
    \underline{R}_b^0(\widetilde B^n)_i + \delta_t q_1^n \cdot \underline{\mathfrak s}_i^{(b,0)} + \delta_t q_N^n\cdot \underline{\mathfrak s}_i^{(b,\ell)} + \frac{\Delta x^2}{2\Delta t}\frac{q_{i+1}^n - 2q_i^n + q_{i-1}^n}{\Delta x^2}\mathds 1_{i\not\in\{1,N\}} , \nonumber
\end{align}
where it remains to specify how to compute $\delta_t q_1^n$ and $\delta_t q_N^n$, and where the vectors $\underline{\mathfrak s}^{(b,0)}$ and $\underline{\mathfrak s}^{(b,\ell)}$ are taken as the solutions $v\in\mathbb R^N$ of
\begin{gather}\label{eq:s-discr}
    v_i - \delta_x\Big(\frac{h_b^3}{3}\delta_x\Big(\frac{v}{h_b}\Big)\Big)_i + \frac{h_{b,i}}{2\Delta x^2}(b_{i+1} - 2b_i + b_{i-1}) v_i = 0 ,\quad 2\leq i\leq N-1
\end{gather}
with respective conditions
\begin{gather}\label{eq:bndcond-s}
    \begin{cases}
    \underline{\mathfrak s}_1^{(b,0)} = 1 \\
    \underline{\mathfrak s}_N^{(b,0)} = 0
    \end{cases} ,\qquad
    \begin{cases}
    \underline{\mathfrak s}_1^{(b,\ell)} = 0 \\
    \underline{\mathfrak s}_N^{(b,\ell)} = 1
    \end{cases} ,
\end{gather}
yielding a discrete counterpart of~\eqref{defs0lb}. Owing to the conditions~\eqref{eq:bndcond-s}, to the definition~\eqref{eq:Rtheta0-interior} of $\underline R_b^0$ and to Remark~\ref{rem:hbS-kernel},
the update~\eqref{eq:q-O1-update} in border cells $i = 1,N$ reduces to
\begin{equation}\label{eq:discrvar-qborder}
\frac{q_1^{n+1} - q_1^n}{\Delta t} = \delta_t q_1^n,\qquad
\frac{q_N^{n+1} - q_N^n}{\Delta t} = \delta_t q_N^n .
\end{equation}
\begin{remark}\label{rem:hbSstar-to-laxfriedrichs}
We also want to comment on the last term in the right hand side of~\eqref{eq:q-O1-update}, which is a first order numerical diffusion required for stability purposes. We can then show that the update~\eqref{eq:q-O1-update} amounts to a finite volumes method with Lax-Friedrichs nonlocal numerical flux and with source terms accounting for the bathymetry and the boundary layer. In fact, defining the nonlocal flux vector as
\begin{equation}\label{eq:nonlocal-flux-vector}
    \mathfrak f^n = \Big(\frac{h_{b,i}^2}{\alpha_{b,i}}\underline R_b^1\Big(\frac{1}{h_b^2}\widetilde f_{\mathrm{NSW}}^n\Big)_i\Big)_{1\leq i\leq N} ,
\end{equation}
and the associated Lax-Friedrichs flux
\begin{align}\label{eq:nonlocal-numflux}
    \mathfrak f_{i+1/2}^n = \frac{1}{2}(\mathfrak f_{i}^n + \mathfrak f_{i+1}^n) - \frac{\Delta x}{2\Delta t}(q_{i+1}^n - q_i^n) ,
\end{align}
we can show that
for $2\leq i\leq N-1$ the update~\eqref{eq:q-O1-update} equivalently rewrites
\begin{align}\label{eq:q-O1-FVupdate}
&
    \frac{q_i^{n+1} - q_i^n}{\Delta t} + \frac{\mathfrak f_{i+1/2}^n - \mathfrak f_{i-1/2}^n}{\Delta x} ={} \\
    & \quad
    \underline{R}_b^0(\widetilde B^n)_i - \Big(\frac{\delta_x(\alpha_b)_i}{\alpha_{b,i}} + \frac{3}{2}\frac{\delta_x (b)_i}{h_{b,i}}\Big) \mathfrak f_i^n
    + \delta_t q_1^n \cdot \underline{\mathfrak s}_i^{(b,0)} + \delta_t q_N^n\cdot \underline{\mathfrak s}_i^{(b,\ell)} . \nonumber
\end{align}
When the bottom is flat, the vectors $\delta_x(b)$ and $\widetilde B^n$ cancel, so that the discharge update~\eqref{eq:q-O1-FVupdate} is conservative up to the boundary cells,
which is coherent with the continuous model. Note also that when considering the scheme for the Boussinesq-Abbott model with flat topography ($b\equiv0$), we can directly use the expression~\eqref{eq:s-sinh} to compute  $\underline{\mathfrak s}^{(b,0)}$ and $\underline{\mathfrak s}^{(b,\ell)}$ instead of approximating them through~\eqref{eq:s-discr}-\eqref{eq:bndcond-s}. However in practice this does not affect the numerical results in a noticeable way.
\end{remark}

Finally we detail the handling of the border cells and the computation of $\delta_t q_1^n,\delta_t q_N^n$ in the framework of general boundary conditions discussed in Section~\ref{subsec:generic-bnd}. Assuming Assumption~\ref{assH} holds, the knowledge of $U_1^{n+1},U_N^{n+1}$ is deduced from the reconstruction formulas~\eqref{eq:H-reconstruction} reading as
\[
U_1^{n+1} = {\mathcal H}_0(\xi^+_0(U_1^{n+1}),\xi^-_0(U_1^{n+1}))
\quad\mbox{ and }\quad
U_N^{n+1} = {\mathcal H}_\ell(\xi^+_\ell(U_N^{n+1}),\xi^-_\ell(U_N^{n+1})).
\]
In the above, we set the input functions according to~\eqref{eq:generic-bnd} as $\xi_0^+(U_1^{n+1}) = g_0^{n+1}$ and $\xi_\ell^-(U_N^{n+1}) = g_\ell^{n+1}$. On the other hand, the output values $\xi_0^-(U_1^{n+1})$ and $\xi_\ell^+(U_N^{n+1})$ are obtained by discretizing the ODE~\eqref{CBODE-recall} for the traces as follows.
Given $X^n$ and $Y^n$ the respective approximations at time $t^n$ of the functions
\begin{equation}
    X:t \longmapsto \begin{pmatrix}
    \xi_0^-(U(t,0)) \\
    \xi_\ell^+(U(t,\ell))
    \end{pmatrix} \in\mathbb R^2,\qquad
    Y:t \longmapsto \dot X(t),
\end{equation}
we define the update $X^{n+1},Y^{n+1}$ as
\begin{align} \label{eq:update-Ynp1}
\begin{cases}
\cfrac{X^{n+1} - X^n}{\Delta t} = Y^{n+1} \\
\widetilde{D}_1^n\cfrac{Y^{n+1} - Y^{n}}{\Delta t} +
\underline{\mathfrak S}'_b\widetilde{D}_2^nY^{n+1} +
\begin{pmatrix}
{\mathcal Q}_1^n(\frac{g_0^{n+1} - g_0^{n-1}}{2\Delta t},Y_1^n) \\
{\mathcal Q}_2^n(Y_2^n,\frac{g_\ell^{n+1} - g_\ell^{n-1}}{2\Delta t})
\end{pmatrix} +  \widetilde{V}_{\rm bdry}^n \\
\hspace{6em} = \widetilde{V}_{\rm int}^n
 -
 {D}_1^n\begin{pmatrix}  \frac{g_0^{n+1} - 2g_0^n + g_0^{n-1}}{\Delta t^2} \\ \frac{g_\ell^{n+1} - 2g_\ell^n + g_\ell^{n-1}}{\Delta t^2} \end{pmatrix}
-
\underline{\mathfrak S}'_b{D}_2^n\begin{pmatrix} \frac{g_0^{n+1} - g_0^{n-1}}{2\Delta t} \\ \frac{g_\ell^{n+1} - g_\ell^{n-1}}{2\Delta t} \end{pmatrix}
\end{cases}
\end{align}
where the diagonal matrices $D_j^n, \widetilde D_j^n$ are defined for $j\in\{1,2\}$ as
\begin{gather*}
D_j^n=\begin{pmatrix} \partial_1{\mathcal H}_{0,j}(g_0^n,X_1^n) & 0 \\ 0 & \partial_2 {\mathcal H}_{\ell,j}(X_2^n,g_\ell^n) \end{pmatrix} ,\\
\widetilde{D}_j^n=\begin{pmatrix} \partial_2{\mathcal H}_{0,j}(g_0^n,X_1^n) & 0 \\ 0 & \partial_1 {\mathcal H}_{\ell,j}(X_2^n,g_\ell^n) \end{pmatrix},
\end{gather*}
where we used the quadratic forms
\begin{align*}
{\mathcal Q_1^n}(u_0,v_0)&= 
 \begin{pmatrix} u_0 \\ v_0 \end{pmatrix}\cdot {\rm Hess}_{{\mathcal H}_{0,1}}(g_0^n,X_1^n) \begin{pmatrix} u_0 \\ v_0\end{pmatrix},\\
 {\mathcal Q_2^n}(u_\ell,v_\ell)&= 
 \begin{pmatrix} u_\ell \\ v_\ell \end{pmatrix}\cdot {\rm Hess}_{{\mathcal H}_{\ell,1}}(X_2^n,g_\ell^n) \begin{pmatrix} u_\ell \\ v_\ell\end{pmatrix} ,
\end{align*}
and where
\begin{gather*}\label{defSbprime-discrete}
\underline{\mathfrak S}_b'=\begin{pmatrix}
\delta_x(\underline{\mathfrak s}^{(b,0)})_1  & \delta_x(\underline{\mathfrak s}^{(b,\ell)})_1 \\
\delta_x(\underline{\mathfrak s}^{(b,0)})_N  & \delta_x(\underline{\mathfrak s}^{(b,\ell)})_N
\end{pmatrix} ,\qquad
\widetilde{V}_{\rm bdry}^n =
\begin{pmatrix}
 \frac{3}{h_{b,1}^2}\widetilde{f}_{\rm NSW} ({\mathcal H}_0(g_0^n,X_1^n)) \\
\frac{3}{h_{b,N}^2}\widetilde{f}_{\rm NSW} ({\mathcal H}_\ell(X_2^n,g_\ell^n)) 
 \end{pmatrix} \\
  \widetilde V_{\rm int}^n=
\begin{pmatrix}
3(\underline R_b^1 (\frac{1}{h_b^2}\widetilde f_{\rm NSW}^n) )_{1}    - \big(\delta_x (\underline R^0_b \widetilde B^n) \big)_{1}  \\
3(\underline R_b^1 (\frac{1}{h_b^2}\widetilde f_{\rm NSW}^n) )_{N}   -\big( \delta_x (\underline R^0_b \widetilde B^n)\big)_{N}
 \end{pmatrix} .
\end{gather*}
The functions $g_0,g_\ell$ are defined over $t\in\mathbb R_+$, therefore the terms $g_0^{n-1},g_\ell^{n-1}$ found in~\eqref{eq:update-Ynp1} are not known when $n = 0$. We suggest to define them following the relations below:
\begin{gather*}
    \frac{g_0(\Delta t) - g_0^{-1}}{2\Delta t} = \frac{-3g_0(0) + 4g_0(\Delta t) - g_0(2\Delta t)}{2\Delta t} , \\
    \frac{g_\ell(\Delta t) - g_\ell^{-1}}{2\Delta t} = \frac{-3g_\ell(0) + 4g_\ell(\Delta t) - g_\ell(2\Delta t)}{2\Delta t}.
\end{gather*}
The sequence $(X^n)_n$ is initialized by taking $X^0 = (\xi_0^-(U_1^0), \xi_\ell^+(U_N^0))^T$. As in the continuous case (see the proof of Theorem \ref{theoWP}), when $\partial_2\mathcal H_{0,1}(g_0^0,X_1^0)\neq 0$ (respectively when $\partial_2\mathcal H_{0,1}(X_2^0,g_\ell^0)\neq 0$), an initial value $Y_1^0$ (respectively $Y_2^0$) must also be provided in order to compute the second line of~\eqref{eq:update-Ynp1} for $n=0$. These quantities can be obtained by discretizing~\eqref{CIODE2} in the following way
\begin{equation}\label{CIODE2-discr}
\begin{cases}
\partial_2{\mathcal H}_{0,1} (g_0(0),X_1^0)  \, Y_1^0 &=  -\delta_x (q^0)_1 - \partial_1{\mathcal H}_{0,1}(g_0(0),X_1^0)\cfrac{g_0(\Delta t) - g_0(0)}{\Delta t} , \\[.25em]
\partial_1{\mathcal H}_{\ell,1} (X_2^0,g_\ell(0)) \, Y_2^0 &= -\delta_x (q^0)_N - \partial_2{\mathcal H}_{\ell,1}(X_2^0,g_\ell(0))\cfrac{g_\ell(\Delta t) - g_\ell(0)}{\Delta t} .
\end{cases}
\end{equation}
Remark that in~\eqref{CIODE2-discr}, if for a given index $1\leq k\leq 2$ the coefficient in factor of $Y_{k}^0$ cancels, then the quadratic function $\mathcal Q_k^0$ does not depend on $Y_{k}^0$ anymore (this is the discrete counterpart of Remark \ref{remvan}); hence in the second equation of~\eqref{eq:update-Ynp1} the value $Y_k^0$ is not required to compute $Y_k^1$. The proposed discretization~\eqref{eq:update-Ynp1} thus automatically adapts to the order of the ODE for the output functions, which dictates as in the continuous case how many initial conditions are required on the output functions.

Once the approximation $X^{n+1}$ of the output functions is known, it is possible to compute the border values $U_1^{n+1} = \mathcal H(g_0^{n+1},X_1^{n+1})$ and $U_N^{n+1} = \mathcal H(X_2^{n+1},g_\ell^{n+1})$, and $\delta_t q_1^{n},\delta_t q_N^n$ are deduced from~\eqref{eq:discrvar-qborder} so that we can update the interior discharge through~\eqref{eq:q-O1-FVupdate}. The iteration is complete.

\begin{remark}
We need to check that the update~\eqref{eq:update-Ynp1} becomes well-defined for $\Delta t > 0$ small enough. Indeed, the second equation of this system admits a unique solution $Y^{n+1}$ if the matrix $\widetilde D_1^n + \Delta t \underline{\mathfrak S}_b' \widetilde D_2^n$ is invertible. Its determinant is the second degree polynomial in $\Delta t$ given by
\begin{gather*}
    \det(\widetilde D_1^n) + \Delta t\Big((\underline{\mathfrak S}_b')_{1,1} (\widetilde D_1^n)_{2,2} (\widetilde D_2^n)_{1,1} + (\underline{\mathfrak S}_b')_{2,2} (\widetilde D_1^n)_{1,1} (\widetilde D_2^n)_{2,2}\Big) \\*
    {}+ \Delta t^2 \Big(\det(\underline{\mathfrak S}_b') \det(\widetilde D_2^n)\Big); 
\end{gather*}
as a consequence of Assumption~\ref{assH2} and Proposition \ref{propS'} this determinant is non zero for $\Delta t$ small enough, so that the update~\eqref{eq:update-Ynp1} is well-defined.
\end{remark}

When approximating the nonlinear Boussinesq-Abbott system~\mbox{\eqref{eq:BP-reformulated2}-\eqref{eq:BP-boundarylayer2}} with the scheme described above, in practice we need to restrict the time-step with a CFL condition similar to that usually encountered when discretizing the hyperbolic nonlinear shallow water system, that is to say we take
\[
\frac{\Delta t}{\Delta x} \leq \frac{K}{\max_i \lvert\lambda(U_i^n)\lvert} ,
\]
with $\lvert\lambda(U_i^n)\rvert = \lvert u_i^n\rvert + \sqrt{\mathtt gh_i^n}$ the maximum eigenvalue in absolute value of the Jacobian matrix $Df_{\rm NSW}(U_i^n)$.
In our numerical experiments such a constraint seems to be a requirement to achieve stable results.

\subsection{Second order MacCormack scheme}

\label{subsec:2nd-order-discr}
Due to the dispersive nature of the Boussinesq-type models, it can be challenging to increase the order of their numerical approximations while keeping them stable and free of spurious oscillations in the presence of non homogeneous boundary conditions. A strategy introduced in \cite{beck2023numerical} for the numerical simulation of waves interacting with floating objects is the MacCormack discretization method, which we describe in the lines below. It consists in alternating a first order prediction step with a correction step during which the predicted state is used. The final update is then obtained by averaging the prediction and correction steps. This procedure is explained in detail thereafter.

\medbreak\noindent
\textsc{Prediction step}~---~Given an approximation $(U_i^n)_{1\leq i\leq N}$ of the solution to~\mbox{\eqref{eq:BP-reformulated2}-\eqref{eq:BP-boundarylayer2}} at time $t^n$, we first compute $U_{\mathrm P,i}^{n+1} = (\zeta_{\mathrm P,i}^{n+1},q_{\mathrm P,i}^{n+1})^T$ the predicted state using a forward Euler method with a left upwinding for the numerical fluxes in every interior cell $2\leq i\leq N-1$
\begin{align}\label{eq:mccormack-prediction}
    \begin{cases}
    \displaystyle
    \frac{\zeta_{\mathrm P,i}^{n+1} - \zeta_i^n}{\Delta t} + \frac{1}{\Delta x}\Big(q_{i}^n - q_{i-1}^n \Big) = 0 \\[.75em]
    \displaystyle
    \frac{q_{\mathrm P,i}^{n+1} - q_i^n}{\Delta t} + \frac{\mathfrak f_{i}^n - \mathfrak f_{i-1}^n}{\Delta x} = \underline{R}_b^0(\widetilde B^n)_{i-1} + \delta_t q_{\mathrm P,1}^n \cdot \underline{\mathfrak s}_{i-1}^{(b,0)} + \delta_t q_{\mathrm P,N}^n\cdot \underline{\mathfrak s}_{i-1}^{(b,\ell)} \\[.75em]
    \displaystyle
    \hspace{11.25em} {}- \frac{1}{2}\Big(\frac{\alpha_{b,i+1} - \alpha_{b,i-1}}{\alpha_{b,i}\, \Delta x} + \frac{3}{2}\frac{b_{i+1} - b_{i-1}}{h_{b,i}\, \Delta x}\Big) \mathfrak f_{i-1}^n
    \end{cases} ,
\end{align}
where we kept the same definitions for $\underline R_b^0,\underline R_b^1,\mathfrak f^n,\widetilde B^n,\underline{\mathfrak s}^{(b,0)},\underline{\mathfrak s}^{(b,\ell)},\alpha_b$ as in the previous Section, which involve the second order centered operator $\delta_x$ given in~\eqref{eq:deltax}.
The boundary cells are treated similarly to the first order scheme described in Section~\ref{subsec:1st-order-discr}; that is to say we enforce
\[
\xi_0^+(U_{\mathrm P,1}^{n+1}) = g_0^{n+1}, \qquad \xi_\ell^-(U_{\mathrm P,N}^{n+1}) = g_\ell^{n+1} ,
\]
then compute the output functions according to~\eqref{eq:update-Ynp1} to get $(X,Y)_{\mathrm P}^{n+1}$, and finally use the reconstruction formulas~\eqref{eq:H-reconstruction} to get $U_{\mathrm P,1}^{n+1}, U_{\mathrm P,N}^{n+1}$ and define
\[
\delta_t q_{\mathrm P,1}^n = \frac{q_{\mathrm P,1}^{n+1} - q_{1}^n}{\Delta t} , \qquad
\delta_t q_{\mathrm P,N}^n = \frac{q_{\mathrm P,N}^{n+1} - q_{N}^n}{\Delta t} .
\]
Notice that, although in~\eqref{eq:mccormack-prediction} we did not write the discharge update under a form involving an explicit discretization of $(h_b S)^*$, it is of course possible to rewrite it in a form similar to~\eqref{eq:q-O1-update}, where $(\underline{h_b S})^*$ defined in~\eqref{eq:hbS-star-discrete} has to be substituted with its left-upwinded counterpart, and without the numerical diffusion term. A similar comment will hold in the correction step.

\medbreak\noindent
\textsc{Correction step}~---~Next the correction term $U_{\mathrm C,i}^{n+1} = (\zeta_{\mathrm C,i}^{n+1},q_{\mathrm C,i}^{n+1})^T$ is computed using a forward Euler update with right upwinding in the interior cells $2\leq i\leq N-1$. It involves the predicted state obtained in the previous step in order to evaluate the flux and source terms:
\begin{align}\label{eq:mccormack-correction}
    \begin{cases}
    \displaystyle
    \frac{\zeta_{\mathrm C,i}^{n+1} - \zeta_i^n}{\Delta t} + \frac{1}{\Delta x}\Big(q_{\mathrm P,i+1}^{n+1} - q_{\mathrm P,i}^{n+1} \Big) = 0 \\[.75em]
    \displaystyle
    \frac{q_{\mathrm C,i}^{n+1} - q_i^n}{\Delta t} + \frac{\mathfrak f_{\mathrm P,i+1}^{n+1} - \mathfrak f_{\mathrm P,i}^{n+1}}{\Delta x} = \underline{R}_b^0(\widetilde B_{\mathrm P}^{n+1})_{i+1} + \delta_t q_{\mathrm C,1}^n \cdot \underline{\mathfrak s}_{i+1}^{(b,0)} + \delta_t q_{\mathrm C,N}^n\cdot \underline{\mathfrak s}_{i+1}^{(b,\ell)} \\[.75em] 
    \displaystyle
    \hspace{13em} {}- \frac{1}{2}\Big(\frac{\alpha_{b,i+1} - \alpha_{b,i-1}}{\alpha_{b,i}\, \Delta x} + \frac{3}{2}\frac{b_{i+1} - b_{i-1}}{h_{b,i}\, \Delta x}\Big) \mathfrak f_{\mathrm P,i+1}^{n+1}
    \end{cases} \! .
\end{align}
In the above, the fluxes $\mathfrak f_{\mathrm P,i}^{n+1}$ and the source term $\widetilde B_{\mathrm P,i}^{n+1}$ are defined as in~\eqref{eq:nonlocal-flux-vector} and~\eqref{eq:discr-flux-source} but making use of the prediction state, that is to say
\begin{gather*}
\forall 1\leq i\leq N,\qquad
\begin{cases}
    \displaystyle
    \widetilde B_{\mathrm P,i}^{n+1} &= \displaystyle -{\mathtt g} \zeta_{\mathrm P,i}^{n+1} \delta_x(b)_i + \frac{3}{2}\frac{\delta_x(b)_i}{h_{b,i}}\widetilde{f}_{\mathrm{NSW}}(U_{\mathrm P,i}^{n+1}) , \\
    \displaystyle
    \mathfrak f_{\mathrm P,i}^{n+1} &= \displaystyle \frac{h_{b,i}^2}{\alpha_{b,i}}\underline R_b^1\Big(\Big\{\frac{1}{h_{b,j}^2}\widetilde f_{\mathrm{NSW}}(U_{\mathrm P,j}^{n+1})\Big\}_{1\leq j\leq N}\Big)_i.
\end{cases}
\end{gather*}
The handling of the border cells is the same as before, except that in~\eqref{eq:update-Ynp1}, for $k\in\{1,2\}$ the matrices $D_k,\widetilde D_k$ and the quadratic forms $\mathcal Q_k^n$ need to be evaluated with the predicted values $(X,Y)_{\mathrm P}^{n+1}$ obtained in the previous step instead of $(X,Y)^n$, and the vectors $\widetilde{V}_{\rm bdry}^n, \widetilde{V}_{\rm int}^n$ need to be evaluated with $U_{\mathrm P}^{n+1}$ instead of $U^n$. Therefore the correction update for the output functions can be written
\begin{align} \label{eq:update-Ynp1-correction}
\begin{cases}
\cfrac{X_{\mathrm C}^{n+1} - X^n}{\Delta t} = Y_{\mathrm C}^{n+1} \\
\widetilde{D}_{\mathrm P,1}^{n+1}\cfrac{Y_{\mathrm C}^{n+1} - Y^{n}}{\Delta t} +
\underline{\mathfrak S}'_b\widetilde{D}_{\mathrm P,2}^{n+1} Y_{\mathrm C}^{n+1} +
\begin{pmatrix}
{\mathcal Q}_{\mathrm P,1}^{n+1}(\frac{g_0^{n+1} - g_0^{n-1}}{2\Delta t},Y_{\mathrm P,1}^{n+1}) \\
{\mathcal Q}_{\mathrm P,2}^{n+1}(Y_{\mathrm P,2}^{n+1},\frac{g_\ell^{n+1} - g_\ell^{n-1}}{2\Delta t})
\end{pmatrix} +  \widetilde{V}_{\rm bdry, P}^{n+1} \\
\hspace{6em} = \widetilde{V}_{\rm int, P}^{n+1}
 -
 {D}_{\mathrm P,1}^{n+1}\begin{pmatrix}  \frac{g_0^{n+1} - 2g_0^n + g_0^{n-1}}{\Delta t^2} \\ \frac{g_\ell^{n+1} - 2g_\ell^n + g_\ell^{n-1}}{\Delta t^2} \end{pmatrix}
-
\underline{\mathfrak S}'_b{D}_{\mathrm P,2}^{n+1}\begin{pmatrix} \frac{g_0^{n+1} - g_0^{n-1}}{2\Delta t} \\ \frac{g_\ell^{n+1} - g_\ell^{n-1}}{2\Delta t} \end{pmatrix}
\end{cases} \! ,
\end{align}
and after computing $U_{\mathrm C,1}^{n+1} = \mathcal H(g_0^{n+1},X_{\mathrm C,1}^{n+1})$ and $U_{\mathrm C,N}^{n+1} = \mathcal H(X_{\mathrm C,2}^{n+1},g_\ell^{n+1})$ we set
\[
\delta_t q_{\mathrm C,1}^n = \frac{q_{\mathrm C,1}^{n+1} - q_{1}^n}{\Delta t} , \qquad
\delta_t q_{\mathrm C,N}^n = \frac{q_{\mathrm C,N}^{n+1} - q_{N}^n}{\Delta t} .
\]

\medbreak\noindent
\textsc{Final step}~---~The final update is defined as the average between the prediction and correction states for all $2\leq i\leq N-1$, and likewise the output functions are averaged in the border cells $i = 1,N$:
\begin{equation}\label{eq:mccormack-final-update}
\begin{cases}
\zeta_i^{n+1} = \cfrac{1}{2}(\zeta_{\mathrm P,i}^{n+1} + \zeta_{\mathrm C,i}^{n+1}) \\
q_i^{n+1} = \cfrac{1}{2}(q_{\mathrm P,i}^{n+1} + q_{\mathrm C,i}^{n+1})
\end{cases},\quad
\begin{cases}
    \xi_0^-(U_1^{n+1}) = \cfrac{1}{2}(\xi_0^-(U_{\mathrm P,1}^{n+1}) + \xi_0^-(U_{\mathrm C,1}^{n+1})) \\
    \xi_\ell^+(U_N^{n+1}) = \cfrac{1}{2}(\xi_\ell^+(U_{\mathrm P,N}^{n+1}) + \xi_\ell^+(U_{\mathrm C,N}^{n+1}))
\end{cases}
\end{equation}
The final border elevations and discharges $\zeta_1^{n+1}, q_1^{n+1}, \zeta_N^{n+1}, q_N^{n+1}$ are recovered through the formula~\eqref{eq:H-reconstruction} involving the reconstruction maps ${\mathcal H}_0,\mathcal H_\ell$ evaluated respectively with $(g_0(t^{n+1}),\xi_0^-(U_1^{n+1}))$ and $(\xi_\ell^+(U_N^{n+1}),g_\ell(t^{n+1}))$.

\begin{remark}
Nothing prevents us from switching the upwinding direction in the first two steps, so that a right upwinding is performed in the prediction step together with a left upwinding in the correction step. This does not seem to affect the numerical results in practice. What is important is that the direction of upwinding alternates between these two stages, otherwise the method becomes first order in space.
\end{remark}

\begin{remark}
Since we do not make use of ghost cells, the alternation of the upwinding direction characterizing the MacCormack scheme cannot be performed in the border cells, where we always need to upwind towards the interior of the domain. This does not prevent the scheme from reaching a second order of accuracy in space in the border cells. Indeed~\eqref{eq:update-Ynp1},~\eqref{CIODE2-discr} and~\eqref{eq:update-Ynp1-correction} are already second order in space thanks to the use of the discrete operator $\delta_x$ given by~\eqref{eq:deltax} to discretize $\mathfrak S_b', \widetilde V_{\rm int}$.
\end{remark}

\section{Numerical simulations}\label{sectnum}

In this section we aim to validate the proposed first order Lax-Friedrichs and second order MacCormack schemes, and to assess the effect that different boundary conditions can have on the solutions; in particular, we numerically investigate the issue of asymptotic stability of the boundary conditions. To this end we shall consider three cases:
\begin{enumerate}
    \item Input functions given by the elevation, output functions given by the discharge
    \begin{equation*}
        \xi_0^+(\zeta,q) = \xi_\ell^-(\zeta,q) = \zeta,\qquad
        \xi_0^-(\zeta,q) = \xi_\ell^+(\zeta,q) = q;
    \end{equation*}
    \item Input functions given by the discharge, output functions given by the elevation
    \begin{equation*}
        \xi_0^+(\zeta,q) = \xi_\ell^-(\zeta,q) = q,\qquad
        \xi_0^-(\zeta,q) = \xi_\ell^+(\zeta,q) = \zeta;
    \end{equation*}
    \item Input functions given by the incoming Riemann invariants, output functions given by the outgoing Riemann invariants
    \begin{equation*}
        \begin{cases}\xi_0^+(\zeta,q) = u + 2\sqrt{\mathtt gh}\\\xi_\ell^-(\zeta,q) = u - 2\sqrt{\mathtt gh}
        \end{cases},\qquad
        \begin{cases}\xi_0^-(\zeta,q) = u - 2\sqrt{\mathtt gh}\\\xi_\ell^+(\zeta,q) = u + 2\sqrt{\mathtt gh}
        \end{cases}.
    \end{equation*}
\end{enumerate}
The gravitational acceleration $\mathtt g$ will be taken equal to $9.81$, and unless specified otherwise the ratio $\Delta t/\Delta x$ will be taken equal to $0.8$ for the Lax-Friedrichs scheme and to $0.45$ for the MacCormack scheme.

We propose in \S \ref{sectnumres} a numerical resolution of the initial boundary value problem with the different boundary conditions presented above. In the case of a flat bottom, considered in \S \ref{sectnumresflat}, the Boussinesq-Abbott system possesses solitary waves that can be used to study the convergence of the schemes. In the case of a non-flat topography investigated in \S \ref{subsec:BAbottom-testcase}, this is no longer true and we therefore use a solution computed in a wider domain and for a very refined mesh to study the convergence. The theoretical open problem of asymptotic stability is then numerically investigated in \S \ref{sectnumAS}.

\subsection{Numerical resolution of the initial boundary value problem}\label{sectnumres}

\subsubsection{The Boussinesq-Abbott system with flat topography}\label{sectnumresflat}

We consider the case of a solitary wave travelling over a flat bottom, which consists in a solution of~\eqref{eq:BA} of the form
\begin{equation}\label{eq:soliton}
   \zeta(t,x) = \widetilde\zeta(x - x_0 - ct) ,\qquad q(t,x) = \widetilde q(x - x_0 - ct) ,
\end{equation}
where $c$ is the celerity at which the solitary wave propagates without deformation, and where $x_0\in\mathbb R$ is the initial position of the maximum elevation $\zeta(0,x_0)$ denoted $\zeta_{\max}$. Despite the apparent simplicity of this solution, it involves nonlinear and dispersive effects which make it an interesting test-case for numerical validation purposes.
As explained in~\cite{Lannes_Weynans_2020}, the wave profile $\widetilde\zeta$ verifies the following second order ODE \begin{equation}\label{eq:soliton-profile}
   - c^2\frac{h_0\widetilde\zeta}{h_0 + \widetilde\zeta} + \frac{c^2h_0^2}{3}\widetilde\zeta'' + \frac{\mathtt g}{2}(\widetilde\zeta^2 + 2h_0\widetilde\zeta) = 0 .
\end{equation}
It is obtained by remarking that~\eqref{eq:soliton} implies $\widetilde q = c\widetilde\zeta$ from the first equation of~\eqref{eq:BA}. Injecting this in the second equation of~\eqref{eq:BA} and upon integration in space with the hypothesis that $\widetilde\zeta$ vanishes at infinity, one recovers~\eqref{eq:soliton-profile}. Furthermore, keeping in mind that $\widetilde\zeta(0) = \zeta_{\max}$, the celerity is shown to satisfy the relation
\begin{equation*}
    c^2 = \frac{\frac{\mathtt g}{6}\zeta_{\max}^3 + \frac{\mathtt gh_0}{2}\zeta_{\max}^2}{h_0^2(\zeta_{\max}/h_0 - \ln(1 + \zeta_{\max}/h_0))}
\end{equation*}
by integrating~\eqref{eq:soliton-profile} against $\widetilde\zeta'$, using again that $\widetilde\zeta$ cancels at infinity and that $\widetilde\zeta'(0) = 0$.

Once the solution of~\eqref{eq:soliton-profile} has been approximated, it can be used to generate initial and boundary conditions in the domain $(0,\ell)$, and can also serve as a reference solution to measure the error of the proposed numerical schemes. 
The initial condition is taken as $(\zeta^{\rm in}, q^{\rm in})(x) = (\widetilde\zeta,c\, \widetilde\zeta)(x-x_0)$ and the general boundary conditions write
\begin{align*}
    \xi_0^+(\zeta_0,q_0)(t) &= g_0(t) := \xi_0^+(\widetilde\zeta(-x_0-ct), c\, \widetilde\zeta(-x_0-ct)) , \\
    \xi_\ell^-(\zeta_\ell,q_\ell)(t) &= g_\ell(t) := \xi_\ell^-(\widetilde\zeta(\ell-x_0-ct), c\, \widetilde\zeta(\ell-x_0-ct)) .
\end{align*}
In practice we take $\ell = 60$, $h_0 = 1$, $\zeta_{\max} = h_0/5$, and $c > 0$. In order to validate the treatment of the boundary conditions we consider the two settings below.

\begin{description}[leftmargin=1em]
\item[\sc Incoming solitary wave --]The solitary wave is initially centered on $x_0 = -\ell/2$ outside of the computational domain $(0,\ell)$.
It then enters the domain from the left boundary, and the simulation stops at time $t = \ell/c$ when the solitary wave is centered on $\ell/2$.
\item[\sc Outgoing solitary wave --]The solitary wave is initially centered on $x_0 = \ell/2$ inside the computational domain, and the simulation is stopped once most of the wave left the domain through the right boundary at time $t = 3\ell/4c$.
\end{description}

\begin{figure}[htbp!]\centering
\includegraphics[trim={3.5cm 11.82cm 3.5cm 11.125cm}, clip, width=.9\textwidth]{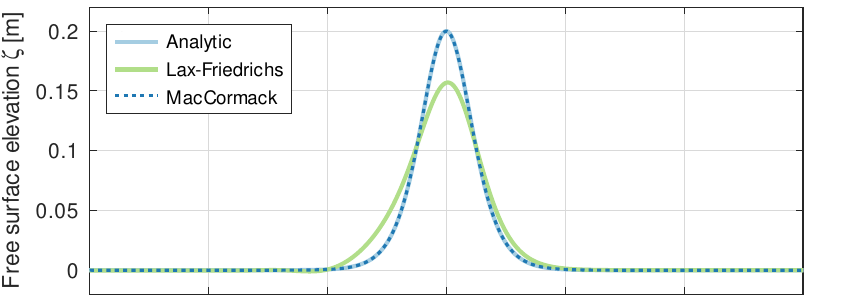}\\
\includegraphics[trim={3.5cm 10.625cm 3.5cm 11.125cm}, clip, width=.9\textwidth]{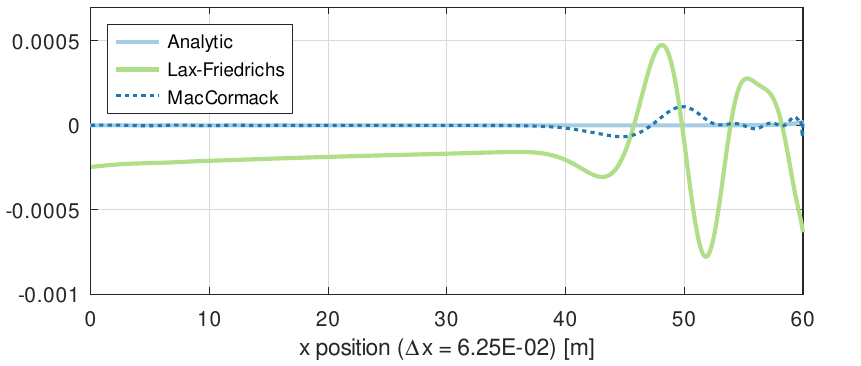}
\caption{Free surface elevation for the incoming and outgoing solitary wave test-cases (respectively top and bottom); incoming Riemann invariants imposed at the boundaries}
\label{fig:incoming-sol-zeta}
\end{figure}

Figure~\ref{fig:incoming-sol-zeta} compares the Lax-Friedrichs and MacCormack schemes at final time for both settings with incoming Riemann invariants enforced as boundary conditions. Unsurprisingly, the Lax-Friedrichs is much more diffusive than MacCormack, and also creates more reflections when trying to evacuate the wave through the right boundary.

Next  we compare the three different choices for $\xi_0^+,\xi_\ell^-$ mentioned at the beginning of Section~\ref{sectnum}, and for each case we display the following $\ell^2$ error on the elevation 
\begin{equation}\label{eq:L2-error}
E_{\rm num}^n = \sqrt{\frac{\Delta x}{\ell}}\bigg(\sum_{i=1}^N \big(\zeta_i^n - \widetilde\zeta(x_i - x_0 - ct^n)\big)^2 \bigg)^{1/2}
\end{equation}
at final time $t=\ell/c$.

Tables~\ref{table:LF-incoming} and~\ref{table:MC-incoming} correspond respectively to the Lax-Friedrichs and MacCormack schemes for the incoming solitary wave test-case.
Experimentally, both methods achieve the expected order of convergence, that is to say first order for Lax-Friedrichs and second order for MacCormack. The difference between the various boundary conditions in terms of the $\ell^2$ error is marginal, except when enforcing the discharge with the MacCormack scheme which leads to a substantially larger error than the two other choices.

\begin{table}[htbp!]\centering
\begin{tabular}{c|cc|cc|cc}
\toprule
$\Delta x$ & \multicolumn{2}{c|}{$\zeta$ enforced} & \multicolumn{2}{c|}{$q$ enforced} & \multicolumn{2}{c}{$R^{\pm}$ enforced}        \\\midrule
& $L^2$-error & Order & $L^2$-error & Order & $L^2$-error & Order \\
  8.82E-02 & 4.409E-03 & --   & 4.728E-03 & --   & 4.482E-03 & --   \\
  6.25E-02 & 3.263E-03 & 0.87 & 3.484E-03 & 0.88 & 3.312E-03 & 0.87 \\
  4.42E-02 & 2.387E-03 & 0.90 & 2.539E-03 & 0.91 & 2.420E-03 & 0.91 \\
  3.12E-02 & 1.727E-03 & 0.93 & 1.833E-03 & 0.94 & 1.750E-03 & 0.93 \\
  2.21E-02 & 1.244E-03 & 0.95 & 1.318E-03 & 0.95 & 1.260E-03 & 0.95 \\
  1.56E-02 & 8.906E-04 & 0.97 & 9.411E-04 & 0.97 & 9.015E-04 & 0.97 \\
\bottomrule
\end{tabular}
\caption{Lax-Friedrichs scheme for the incoming solitary wave with different boundary conditions. In the last column $R^{\pm}$ refers to the incoming Riemann invariant associated with the shallow water system.}
\label{table:LF-incoming}
\end{table}

\begin{table}[htbp!]\centering
\begin{tabular}{c|cc|cc|cc}
\toprule
$\Delta x$ & \multicolumn{2}{c|}{$\zeta$ enforced} & \multicolumn{2}{c|}{$q$ enforced} & \multicolumn{2}{c}{$R^{\pm}$ enforced}        \\\midrule
& $L^2$-error & Order & $L^2$-error & Order & $L^2$-error & Order \\
  1.00E+00 & 7.846E-03 & --   & 1.186E-02 & --   & 9.424E-03 & -- \\
  5.00E-01 & 2.473E-03 & 1.67 & 3.789E-03 & 1.65 & 2.783E-03 & 1.76 \\
  2.50E-01 & 6.670E-04 & 1.89 & 1.026E-03 & 1.89 & 7.199E-04 & 1.95 \\
  1.25E-01 & 1.696E-04 & 1.98 & 2.662E-04 & 1.95 & 1.832E-04 & 1.97 \\
  6.25E-02 & 4.312E-05 & 1.98 & 6.804E-05 & 1.97 & 4.624E-05 & 1.99 \\
  3.12E-02 & 1.107E-05 & 1.96 & 1.731E-05 & 1.97 & 1.162E-05 & 1.99 \\
\bottomrule
\end{tabular}
\caption{MacCormack scheme for the incoming solitary wave.}
\label{table:MC-incoming}
\end{table}

\begin{table}[htbp!]\centering
\begin{tabular}{c|cc|cc|cc}
\toprule
$\Delta x$ & \multicolumn{2}{c|}{$\zeta$ enforced} & \multicolumn{2}{c|}{$q$ enforced} & \multicolumn{2}{c}{$R^{\pm}$ enforced}        \\\midrule
& $L^2$-error & Order & $L^2$-error & Order & $L^2$-error & Order \\
  6.25E-02 & 2.363E-03 & --   & 2.666E-03 & --   & 1.331E-04 & --   \\
  4.42E-02 & 1.785E-03 & 0.81 & 2.031E-03 & 0.78 & 1.070E-04 & 0.63 \\
  3.12E-02 & 1.322E-03 & 0.87 & 1.514E-03 & 0.85 & 8.401E-05 & 0.70 \\
  2.21E-02 & 9.725E-04 & 0.89 & 1.119E-03 & 0.87 & 6.503E-05 & 0.74 \\
  1.56E-02 & 7.053E-04 & 0.93 & 8.145E-04 & 0.92 & 4.939E-05 & 0.79 \\
  1.11E-02 & 5.098E-04 & 0.94 & 5.903E-04 & 0.93 & 3.711E-05 & 0.83 \\
\bottomrule
\end{tabular}
\caption{Lax-Friedrichs scheme for the outgoing solitary wave.}
\label{table:LF-outgoing}
\end{table}

\begin{table}[htbp!]\centering
\begin{tabular}{c|cc|cc|cc}
\toprule
$\Delta x$ & \multicolumn{2}{c|}{$\zeta$ enforced} & \multicolumn{2}{c|}{$q$ enforced} & \multicolumn{2}{c}{$R^{\pm}$ enforced}        \\\midrule
& $L^2$-error & Order & $L^2$-error & Order & $L^2$-error & Order \\
   1.00E+00 & 6.430E-03 & --   & 1.031E-02 & --   & 4.287E-03 & -- \\
   5.00E-01 & 1.885E-03 & 1.77 & 4.437E-03 & 1.22 & 1.993E-03 & 1.11 \\
   2.50E-01 & 5.568E-04 & 1.76 & 1.257E-03 & 1.82 & 5.607E-04 & 1.83 \\
   1.25E-01 & 1.459E-04 & 1.93 & 3.252E-04 & 1.95 & 1.318E-04 & 2.09 \\
   6.25E-02 & 3.674E-05 & 1.99 & 8.223E-05 & 1.98 & 2.826E-05 & 2.22 \\
   3.12E-02 & 9.162E-06 & 2.00 & 2.067E-05 & 1.99 & 7.085E-06 & 2.00 \\
\bottomrule
\end{tabular}
\caption{MacCormack scheme for the outgoing solitary wave.}
\label{table:MC-outgoing}
\end{table}

Likewise, the outgoing solitary wave test-case is addressed in Table~\ref{table:LF-outgoing} for Lax-Friedrichs and Table~\ref{table:MC-outgoing} for MacCormack. The latter method shows good agreement with a second order convergence. For Lax-Friedrichs the situation is comparable to the incoming solitary wave test-case, except that a more refined mesh is required to start approaching a first order convergence. Note also that when enforcing the incoming Riemann invariants with the Lax-Friedrichs scheme, the order of convergence seems smaller than for the other boundary conditions, however we remark that in this case the errors differ by one order of magnitude in favor of the boundary conditions obtained by enforcing the incoming Riemann invariants.

\subsubsection{The Boussinesq-Abbott system with varying bathymetry}
\label{subsec:BAbottom-testcase}

Since no analytic expression is available for solutions of the Boussinesq-Abbott model~\eqref{eq:BP} in the general case of unsteady flows over a non-flat bottom, we instead approximate a reference solution over a large domain $(-\ell, 2\ell)$ with periodic boundary conditions by the mean of a fine mesh; this approximated reference solution is then used to generate boundary conditions for the small domain $(0, \ell)$, in which the solution can be solved numerically with the proposed Lax-Friedrichs and MacCormack schemes.

The small domain length is set to $\ell = 25$, and we choose the following values for the characteristic depth, bathymetry and wave amplitudes:
\begin{align*}
&
    h_0 = 1,\quad\beta = 1/4,\quad A = h_0/4 .
\end{align*}
The bathymetry profile is defined as follows
\begin{align}\label{eq:smooth-bar}
b(x) = \beta h_0\cdot\begin{cases}
1 & \text{if } \lvert x - \ell/2\rvert < \frac{\ell}{4} \\
\sigma(1.5 - \frac{2}{\ell}\lvert x - \ell/2 \rvert)& \text{if } \frac{\ell}{4} \leq\lvert x - \ell/2\rvert\leq \frac{3\ell}{4} \\
0 & \text{otherwise}
\end{cases} ,
\end{align}
where $\sigma:[0,1]\rightarrow[0,1]$ is a smooth step function given here by $\sigma(x) = 6x^5 - 15x^4 + 10x^3$, which allows the bathymetry~\eqref{eq:smooth-bar} to be $C^2$. The initial condition is taken as a Gaussian centered outside of the small domain:
\begin{align}\label{eq:initcond-gaussian}
    \zeta^{\rm in}(x) = A\cdot\exp(-(x - 11)^2/3), \quad
    q^{\rm in}(x) = 5\cdot\zeta^{\rm in}(x) .
\end{align}
Recalling the shallowness parameter $\mu = (2\pi h_0)^2/\lambda^2$ with the characteristic wavelength defined here as $\lambda = \mathrm{length}\, \{x,\ \zeta^{\rm in}(x) > A/100\}$, one has $\mu\approx 0.6$ so that the flow is only moderately shallow.
The nonlinearity parameter $A/h_0$ is $\epsilon = 0.25$.

\begin{figure}[htbp!]\centering
\includegraphics[height=4.9cm,trim={3.95cm 8.5cm 4.75cm 9cm},clip]{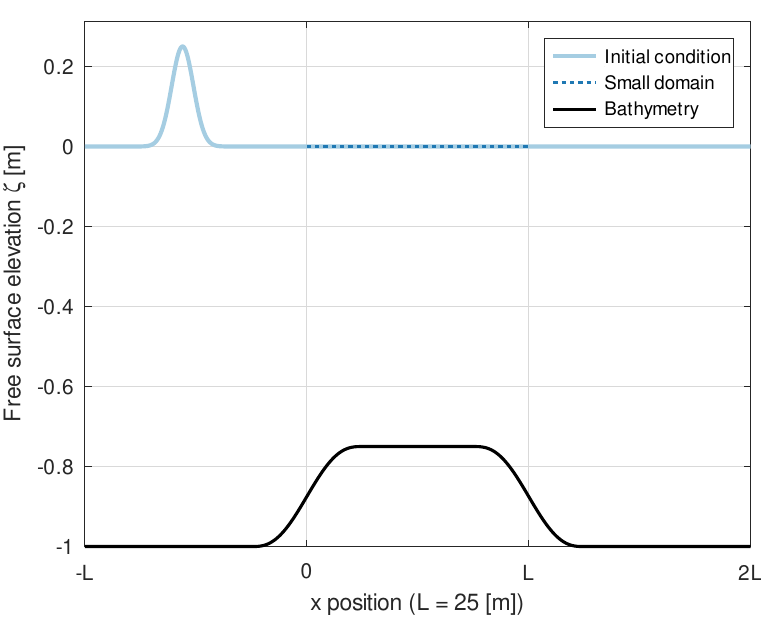}\hfill
\includegraphics[height=4.9cm,trim={4.35cm 8.5cm 4.8cm 9cm},clip]{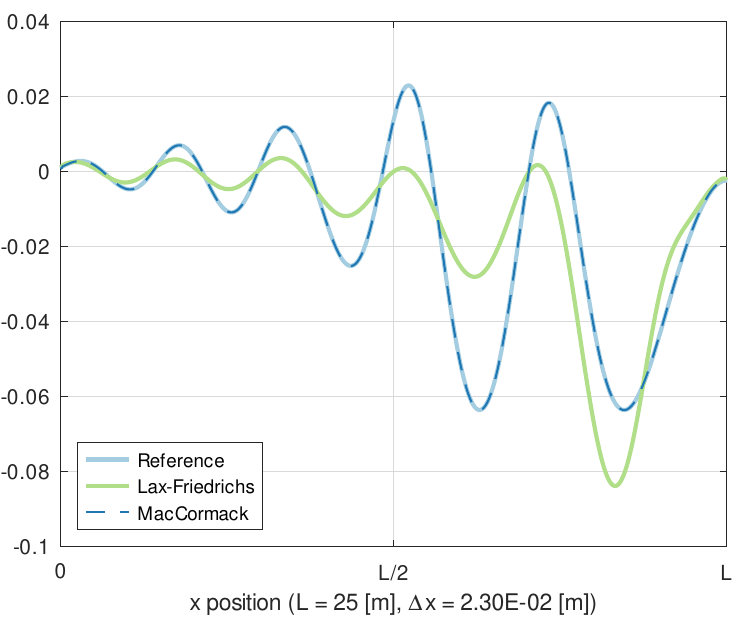}
\caption{Left: initial condition. Right: reference solution, Lax-Friedrichs and MacCormack approximations at time $t = 15$, both obtained using the same CFL constant of 0.45.}
\label{fig:gaussian-over-bump}
\end{figure}

The periodic reference solution is approximated over the large domain $(-\ell,2\ell)$ with a spatial step $\Delta x\approx 8.14\cdot 10^{-3}$, and the simulation stops at time $t = 15$; see Figure~\ref{fig:gaussian-over-bump} for plots of the initial and final states.
Note that the shape of the initial free surface elevation~\eqref{eq:initcond-gaussian} is not preserved through time. More specifically, the Gaussian splits into right- and left-going signals; owing to the periodic condition, both of these signals will enter the small domain $(0,\ell)$ respectively from the left and right boundaries, moreover the right-going signal begins to exit the small domain before the simulation ends. Therefore this can be considered a rich test-case featuring wave generation, interaction and evacuation over a non-flat bottom $b$ whose derivative doesn't vanish at boundaries $x = 0,\ell$.

\begin{table}[htbp!]\centering
\begin{tabular}{c|cc|cc|cc}
\toprule
$\Delta x$ & \multicolumn{2}{c|}{$\zeta$ enforced} & \multicolumn{2}{c|}{$q$ enforced} & \multicolumn{2}{c}{$R^{\pm}$ enforced}        \\\midrule
& $L^2$-error & Order & $L^2$-error & Order & $L^2$-error & Order \\
  3.26E-02 & 1.017E-02 & --   & 1.300E-02 & --   & 8.545E-03 & --   \\
  2.30E-02 & 7.862E-03 & 0.74 & 1.016E-02 & 0.71 & 6.604E-03 & 0.74 \\
  1.63E-02 & 5.962E-03 & 0.80 & 7.759E-03 & 0.78 & 5.003E-03 & 0.80 \\
  1.15E-02 & 4.440E-03 & 0.85 & 5.802E-03 & 0.84 & 3.721E-03 & 0.85 \\
  8.14E-03 & 3.265E-03 & 0.89 & 4.278E-03 & 0.88 & 2.733E-03 & 0.89 \\
  5.75E-03 & 2.375E-03 & 0.92 & 3.118E-03 & 0.91 & 1.986E-03 & 0.92 \\
\bottomrule
\end{tabular}
\caption{Lax-Friedrichs scheme for the Gaussian over bump test-case.}
\label{table:LF-gaussian-varyingbottom}
\end{table}

\begin{table}[htbp!]\centering
\begin{tabular}{c|cc|cc|cc}
\toprule
$\Delta x$ & \multicolumn{2}{c|}{$\zeta$ enforced} & \multicolumn{2}{c|}{$q$ enforced} & \multicolumn{2}{c}{$R^{\pm}$ enforced}        \\\midrule
& $L^2$-error & Order & $L^2$-error & Order & $L^2$-error & Order \\
   2.60E-01 & 4.635E-03 & --   & 5.708E-03 & --   &  3.301E-03 & -- \\
   1.84E-01 & 2.470E-03 & 1.81 & 3.076E-03 & 1.78 &  1.690E-03 & 1.92 \\
   1.30E-01 & 1.288E-03 & 1.89 & 1.629E-03 & 1.84 &  8.485E-04 & 2.00 \\
   9.19E-02 & 6.575E-04 & 1.93 & 8.402E-04 & 1.90 &  4.095E-04 & 2.09 \\
   6.51E-02 & 3.384E-04 & 1.93 & 4.316E-04 & 1.93 &  1.930E-04 & 2.18 \\
   4.60E-02 & 1.777E-04 & 1.85 & 2.211E-04 & 1.92 &  9.056E-05 & 2.17 \\
\bottomrule
\end{tabular}
\caption{MacCormack scheme for the Gaussian over bump test-case.}
\label{table:MC-gaussian-varyingbottom}
\end{table}

Comparing the numerical solution obtained in the small domain with different mesh sizes to the reference solution, we can compute the $\ell^2$ error~\eqref{eq:L2-error} and the associated experimental orders of convergence are shown in Tables~\ref{table:LF-gaussian-varyingbottom} and~\ref{table:MC-gaussian-varyingbottom}. Both the Lax-Friedrichs and MacCormack schemes achieve the expected order for the different boundary conditions considered. The most advantageous choice seems to enforce the incoming Riemann invariants, as it leads to the smallest error.

\subsection{Asymptotic stability}\label{sectnumAS}

Finally we wish to assess the feasibility of achieving a reliable approximation of the solution to the Boussinesq-Abbott system~\eqref{eq:BP} over a varying bathymetry when starting the simulation with a wrong initial condition.
This is motivated by the fact that in real-life, measurements of the free surface elevation can only be performed at a limited number of points, preventing us from knowing the state of the flow everywhere in the domain of interest. Instead, one has to provide an initial guess, which might not be accurate at all.

To this end, we use a setup similar to the one from Section~\ref{subsec:BAbottom-testcase}: a periodic reference solution computed on the large domain $(-\ell,2\ell)$ is used to generate boundary conditions for the small domain $(0,\ell)$. We keep the same domain length $\ell=25$ and the same definition~\eqref{eq:smooth-bar} for the bottom, however the characteristic depth and wave amplitude are now scaled with respect to a parameter $K_\mu$ as follows:
\begin{align*}
    h_0 = \sqrt{K_\mu},\quad
    A = \sqrt{K_\mu}/4 .
\end{align*}
The initial data of the reference solution consists in a sine wave of period $\lambda = 3\ell/5$ and amplitude $A$ for the elevation, together with a nonzero discharge:
\begin{equation}
    \widetilde\zeta^{\rm in}(x) = A\cdot\sin(2\pi x/\lambda),\qquad \widetilde q^{\rm in}(x) = 2\cdot \widetilde\zeta^{\rm in}(x).
\end{equation}
The factor $K_\mu$ shall be taken equal to $(2\pi/\lambda)^{-2}\mu$, with $\mu$ the shallowness parameter valued in $\{0.01, 0.1, 1\}$. The nonlinearity parameter $\epsilon$ and bathymetry parameter $\beta$ are kept to $0.25$. We also introduce the characteristic time $T_0$ given by $2\pi T_0 = \lambda/\sqrt{\mathtt g h_0}$.

\begin{figure}[ht!]\centering
\includegraphics[width=.9\textwidth,trim={4cm 8.5cm 4cm 9.25cm},clip]{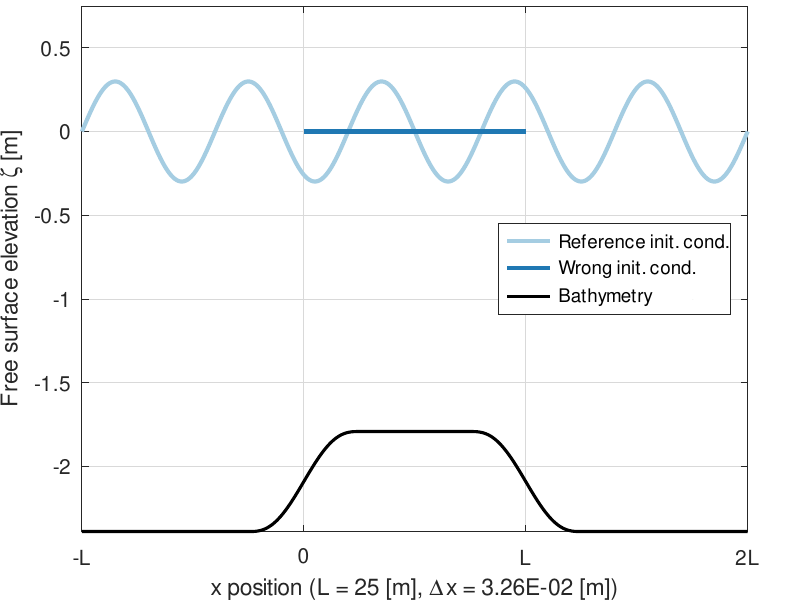}
\caption{Initial elevation for reference and small domain solutions ($\mu=1$).}
\label{fig:enforce-sine-t0}
\end{figure}

When approximating the reference solution over the small domain $(0,\ell)$ with the proposed MacCormack scheme, the simulation is initialized with a lake at rest $(\zeta^{\rm in}, q^{\rm in}) = (0,0)$ which is different from the initial state of the reference solution. This initial setup is plotted in Figure~\ref{fig:enforce-sine-t0} for $\mu = 1$. In order to comply with the system~\eqref{CIODE2}, which we recall provides compatibility conditions when enforcing $\zeta$, the reference solution $(\widetilde\zeta, \widetilde q)$ is progressively enforced at the boundaries of the small domain by the mean of the smooth step function $\sigma$ defined in Section~\ref{subsec:BAbottom-testcase}:
\begin{align*}
&
    \xi_0^+(\zeta_0,q_0)(t) = g_0(t) := \sigma(5t/T_0)\cdot\xi_0^+(\widetilde\zeta(t,0), \widetilde q(t,0)) , \\*
&
   \xi_\ell^-(\zeta_\ell,q_\ell)(t) = g_\ell(t) := \sigma(5t/T_0)\cdot\xi_\ell^-(\widetilde\zeta(t,\ell), \widetilde q(t,\ell)) .
\end{align*}
If on the contrary the boundary conditions on the elevation were enforced directly at $t=0$, then system~\eqref{CIODE2} would be violated, and the well-posedness result from Theorem~\ref{theoWP} does not hold anymore in this situation.

\medbreak
We plot the results obtained at final time $t = 50\, T_0$ in Figure~\ref{fig:sine-over-bump-mu0-mu0p1} for $\mu = 1$ and $\mu = 0.1$, and in Figure~\ref{fig:sine-over-bump-mu0p01} for $\mu = 0.01$. Out of the three boundary conditions tested, only the choice of enforcing the incoming Riemann invariants seems to allow the $\ell^2$ error~\eqref{eq:L2-error} to decay to zero, which is consistent with the considerations of Section~\ref{sectAS}. On the other hand, this is not the case when enforcing the elevation or the discharge; our interpretation is that the initial information --- which represents the perturbation of the reference solution --- is prevented from leaving the domain due to reflections. Reducing the shallowness parameter to~$0.01$, one eventually observes a blow up of the numerical approximation for these two boundary conditions, and which is avoided with the incoming Riemann invariants. As a conclusion, the latter should clearly be preferred over the elevation or discharge in a situation where waves need to be evacuated from the computational domain.

\begin{remark}
In practice, enforcing the incoming Riemann invariant at only one of the two boundaries is enough to recover a good approximation of the reference solution. In this case the $\ell^2$ error on the elevation decays slightly slower towards zero compared to when the incoming Riemann invariants are enforced at both boundaries. 
\end{remark}

\begin{figure}[p!]\centering
\includegraphics[height=4.85cm,trim={3.8cm 8.5cm 4.75cm 9.25cm},clip]{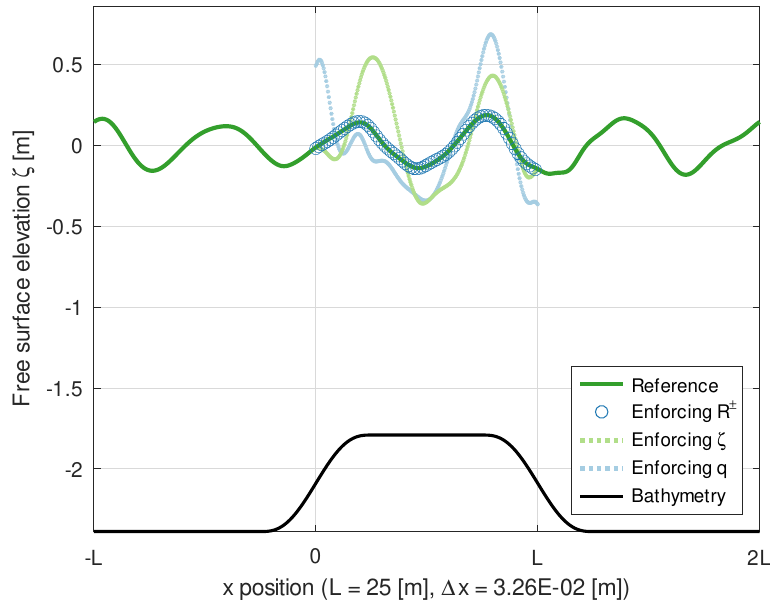}\hfill%
\includegraphics[height=4.85cm,trim={4.5cm 8.5cm 4.75cm 9.25cm},clip]{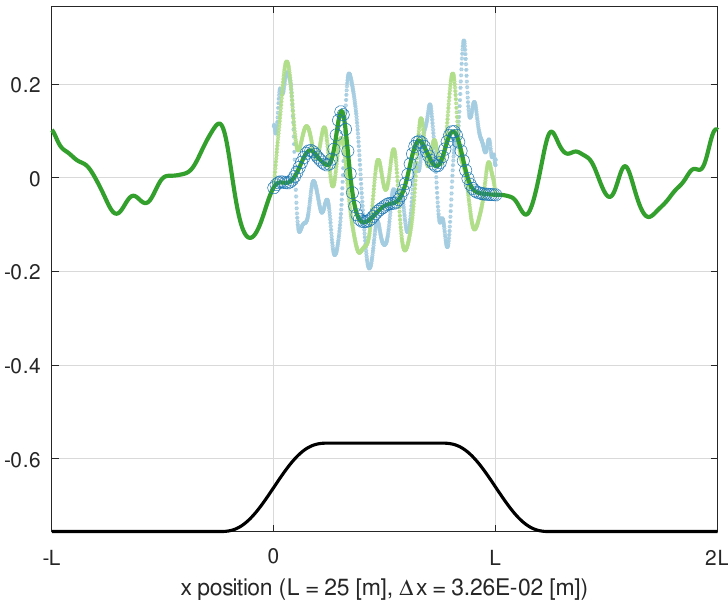}\\[.5em]
\includegraphics[height=4.85cm,,trim={3.8cm 8.5cm 4.75cm 9.25cm},clip]{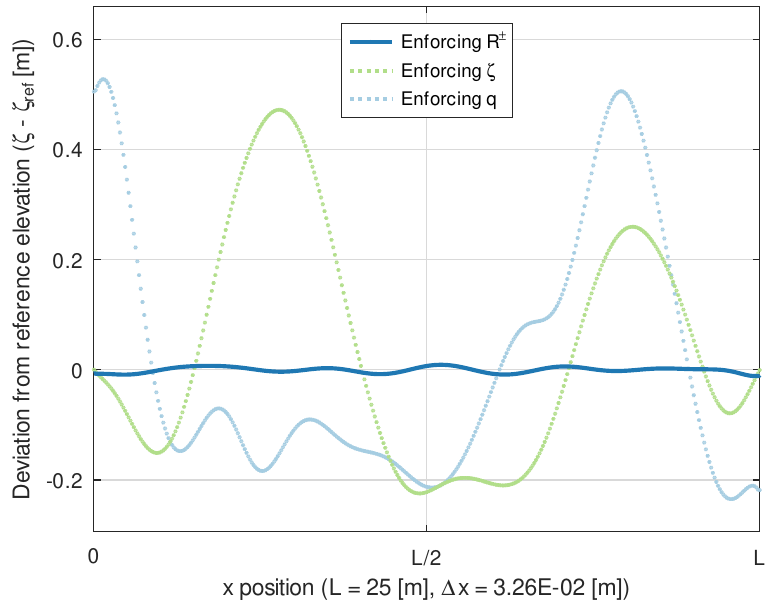}\hfill%
\includegraphics[height=4.85cm,,trim={4.5cm 8.5cm 4.75cm 9.25cm},clip]{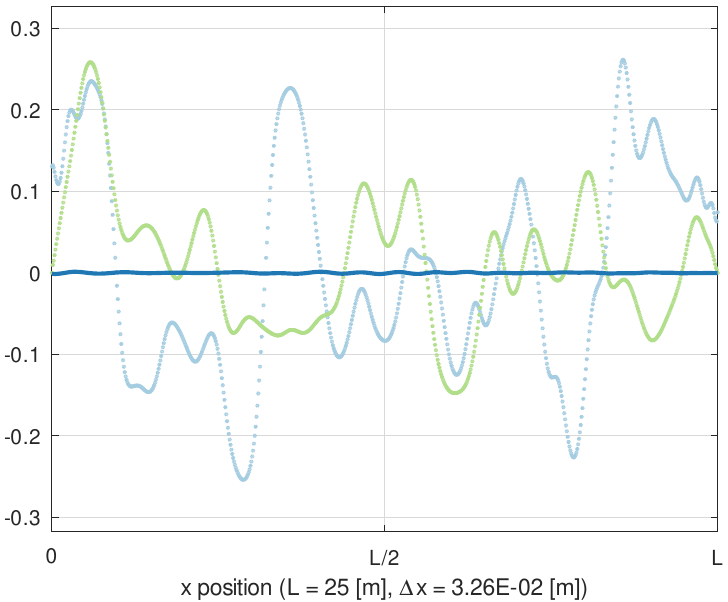}\\[.5em]
\includegraphics[height=4.8cm,trim={3.8cm 8.5cm 4.75cm 9.25cm},clip]{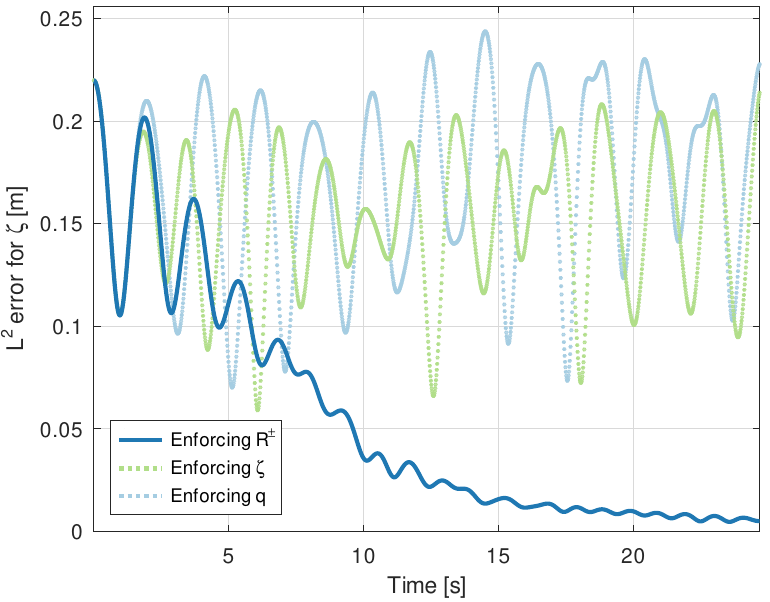}\hfill%
\includegraphics[height=4.8cm,trim={4.375cm 8.5cm 4.75cm 9.25cm},clip]{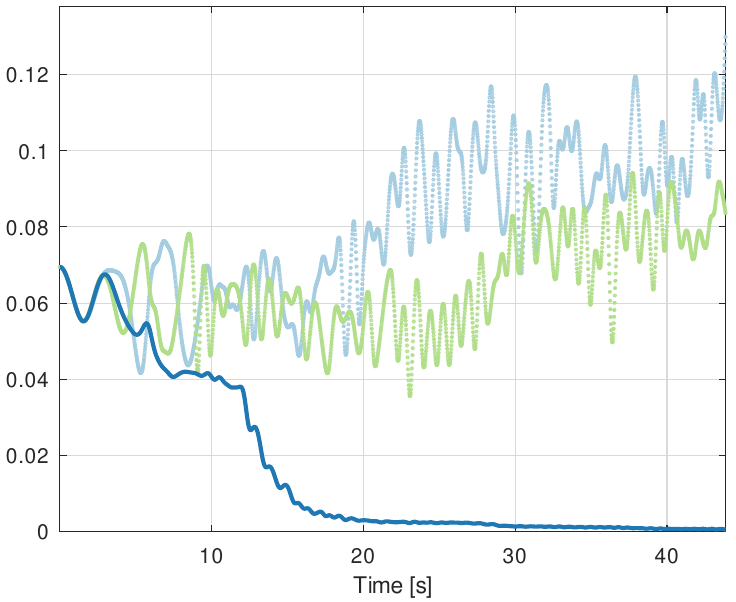}
\caption{Flow over bar test-case for $\mu=1$ (first column) and $\mu=10^{-1}$ (second column). First and second rows correspond respectively to the MacCormack elevation and its deviation from the reference elevation $\widetilde\zeta$ at final time $t = 50\, T_0$ for various boundary conditions. The last row represents the $\ell^2$ error~\eqref{eq:L2-error} over time. Out of the three boundary conditions investigated here, only enforcing the incoming Riemann invariants allows this error to decay towards zero.}
\label{fig:sine-over-bump-mu0-mu0p1}
\end{figure}

\begin{figure}[p!]\centering
\includegraphics[height=7.905cm,trim={3.8cm 8.5cm 4.75cm 9.25cm},clip]{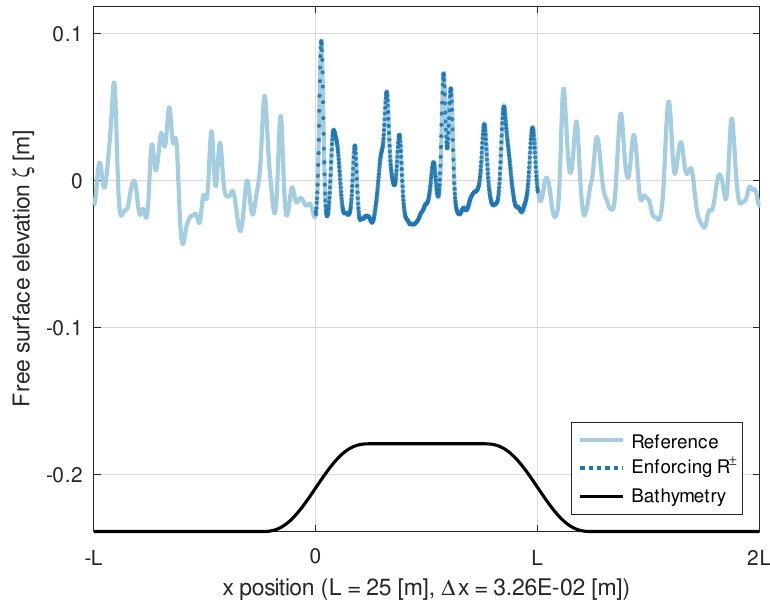}\\[.5em]
\includegraphics[height=7.905cm,trim={3.6cm 8.5cm 4.5cm 9.25cm},clip]{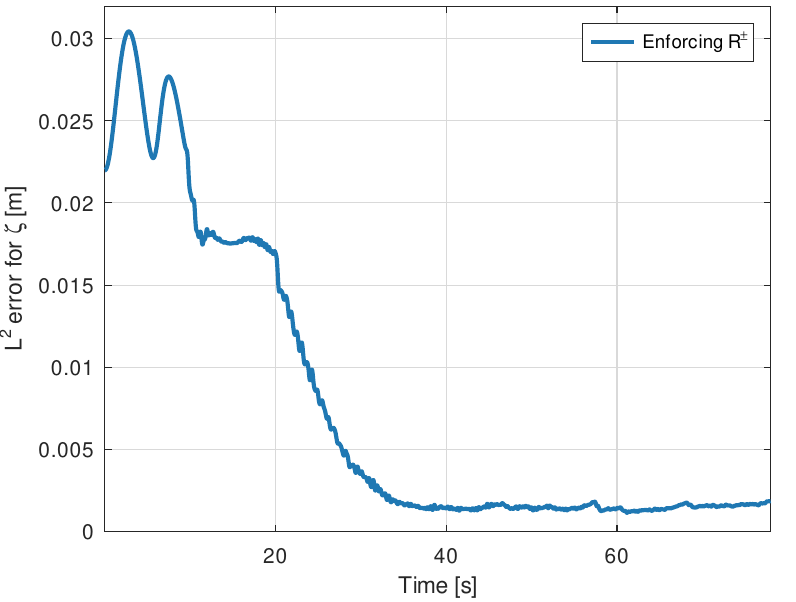}
\caption{Flow over bar test-case for $\mu=10^{-2}$. The first plot was obtained at final time $t = 50\, T_0$. Enforcing $\zeta$ or $q$ as boundary conditions led to a blow up of the MacCormack approximation, therefore only the case of incoming Riemann invariants is displayed here. The $\ell^2$ error~\eqref{eq:L2-error} decays over time but does not quite vanish; this is most likely due to the coarseness of the mesh which make it difficult to accurately capture the high frequencies that develop in the solution. A fix would be to refine the mesh further.}
\label{fig:sine-over-bump-mu0p01}
\end{figure}

\section{Conclusion and perspectives}
In this article, the initial boundary value problem linked to the Boussinesq-Abbott model with non flat bottom has been considered. When imposing the values of the surface elevation at the boundaries, we have seen that like in the case of a flat bottom over a half-line treated in~\cite{Lannes_Weynans_2020}, this problem can be equivalently reformulated as an initial value problem involving a nonlocal flux and an ODE on the trace of the discharge, even in presence of a varying topography and over a bounded domain. This formulation is however more complex because some commutation properties used in the case of a flat bottom are no longer true. We then proposed a new method allowing to prescribe a nonlinear function of the unknowns at the endpoints of the domain by a given function of time. This method is based on the notion of \emph{input} functions, which represent the information that one wants to let in the domain, and that of \emph{output} functions. We believe that the ability to enforce such general boundary conditions is quite relevant for complex applications such as the ones encountered in oceanography. Importantly, our reformulated model defines an ODE on some functional space; under regularity assumptions on the boundary data and on the initial condition, it is possible to obtain the well-posedness of this model in finite time by the Cauchy-Lipschitz theorem. We do not believe however that this approach can be transposed to the Serre-Green-Naghdi model, for which the infinite dimensional ODE structure of the Boussinesq-Abbott system is lost.

This theoretical study was then followed by a numerical discretization of the equations using a hybrid finite difference/finite volumes approach. New schemes of first and second orders were designed using respectively the Lax-Friedrichs numerical flux and the MacCormack strategy, similarly to what was proposed in~\cite{beck2023numerical} for wave-structure interactions over a flat bottom and for boundary conditions on the horizontal discharge. Owing to a simple adaptation of the equations, these schemes can be rendered well-balanced, that is to say that they preserve the lake at rest equilibrium. Next we performed numerical experiments validating the expected orders of convergence for various types of boundary conditions, in a complex setup featuring wave generation and evacuation over a varying topography. Finally we explored numerically the question of asymptotic stability, which is to know whether solutions arising from different initial data but with the same boundary conditions converge to each other after a transitory regime. We found that this is the case when enforcing the incoming Riemann invariants of the shallow water equations, but not when enforcing the elevation nor the discharge. Again this underlines the interest there is to be able to enforce general boundary conditions.

We consider several perspectives to the present work. The first one is related to a limitation inherent to dispersive models in general (including Boussinesq and Serre-Green-Naghdi), that is their inability to describe the breaking of waves; in fact for steep variations of the elevation it is well known that these models produce non physical spurious oscillations. Despite being less accurate than these dispersive models, the shallow water equations present the benefit of being able to handle breaking waves in a physically relevant way by dissipating the energy in the form of shocks, see~\cite{BONNETON20071459} for a physical justification and~\cite{toro2013riemann} for a general methodology on Riemann problems. Note also that the nonlinear shallow water equations are able to take into account the vanishing of the water height $h$ (see~\cite{LannesMetivier2018} for the mathematical analysis of this singularity and~\cite{AUDUSSE2005311} for its numerical  treatment); therefore this model is pertinent to describe the wave dynamics in the surf and swash areas. For these reasons, several authors propose to use a dispersive model and to switch the dispersive terms off in the vicinity of wave breaking~\cite{TonelliPetti,Bonneton2011,KazoleaRicchiuto2018}, hereby working locally with the less precise but more robust nonlinear shallow water equations. The coupling between the nonlinear shallow water equations and dispersive models like the Boussinesq and Serre-Green-Naghdi equation is however not well understood essentially because the boundary conditions that must be imposed on the dispersive component is unclear. A natural perspective of this work is therefore to investigate this coupling.

As pointed in~\cite{FILIPPINI2015109}, the shoaling phenomenon is systematically underestimated (resp. overestimated) by Boussinesq-type models written in elevation-discharge form (resp. in elevation-velocity form). Preliminary tests show that our reformulated Boussinesq-Abbott model is no different; it would thus be interesting to try circumventing this issue. An idea that we would like to explore would consist to perform some form of averaging between the Boussinesq-Abbott model in $(\zeta,q)$ coordinates and the classical Boussinesq-Peregrine model in $(\zeta,v)$ coordinates.

To conclude, we mention a long term goal related to the statistical study of extreme waves. This would require to generate random waves at the boundaries of the domain so as to force a realistic wave field in the domain. We presume that the use of asymptotically stable boundary conditions is crucial in this setting.

\clearpage
\printbibliography

\end{document}